\documentclass[11pt]{amsart}
\usepackage[normalem]{ulem}
\usepackage{mabliautoref}
\usepackage{mathptmx, bbm, amscd, amssymb, enumerate, colonequals, mathdots, xcolor, tikz-cd,mathtools, amsthm}
\usepackage{stmaryrd}
\usepackage{color}
\definecolor{chianti}{rgb}{0.6,0,0}
\definecolor{meretale}{rgb}{0,0,.6}
\definecolor{leaf}{rgb}{0,.35,0}
\include{header.sty}
\DeclareMathOperator{\perfd}{{perfd}}
\newcommand{\bZ}{{\mathbb{Z}}}
\newcommand{\bF}{{\mathbb{F}}}

\definecolor{blush}{rgb}{0.87, 0.36, 0.51}
\definecolor{jazzberryjam}{rgb}{0.65, 0.04, 0.37}
\definecolor{tiffanyblue}{rgb}{0.04, 0.73, 0.71}
\definecolor{darkcyan}{rgb}{0.0, 0.55, 0.55}

\hypersetup{
	bookmarksdepth=3,
	bookmarksopen,
	bookmarksnumbered,
	pdfstartview=FitH,
	colorlinks,
	linkcolor=jazzberryjam,
	anchorcolor=BurntOrange,
	citecolor=MidnightBlue,
	citecolor=tiffanyblue,
	filecolor=darkcyan,
	menucolor=Yellow,
	urlcolor=tiffanyblue
}

\def\spec{\operatorname{Spec}}

\begin{document}
\title[Plus-pure thresholds]{Bounds on the plus-pure thresholds of some hypersurfaces in (ramified) regular rings}
\begin{abstract}
    We study the plus-pure threshold ($\ppt$) of hypersurfaces in mixed characteristic.  We show that the $\ppt$ limits to the $F$-pure threshold (fpt) as we ramify the base DVR.  Additionally, we show that analogs of some positive characteristic extremal singularities cannot attain the same `extremal' $\ppt$ values in the unramified setting.  We also study equations which have controlled ramification when we adjoin their $p$-th roots as well as equations which admit $p$-th roots modulo $p^2$ (or modulo other values), bounding their $\ppt$s. In particular, given a complete unramified regular local ring of mixed characteristic $p>0$, $f^p + p^2 g$ does not define a perfectoid pure singularity for any $f$ and $g$. Finally, we compute bounds on the $\ppt$ of hypersurfaces related to elliptic curves. This gives examples where the $\ppt$ is neither the corresponding $\fpt$ in characteristic $p > 0$ nor the $\lct$ in characteristic zero.  This also provides examples where $p$ times the $\ppt$ is not a jumping number, in stark contrast with the characteristic $p > 0$ picture.
\end{abstract}

\author[Benozzo]{Marta Benozzo}
\address{Laboratoire de Mathématiques d'Orsay, 307 Rue Michel Magat, 91400 Orsay, France}
\email{marta.benozzo@universite-paris-saclay.fr}

\author[Jagathese]{Vignesh Jagathese}
\address{Department of Mathematics, Statistics, and Computer Science, University of Illinois at Chicago, Chicago, IL, USA}
\email{vjagat2@uic.edu}

\author[Pandey]{Vaibhav Pandey}
\address{Department of Mathematics, Purdue University, 150 N University St., West Lafayette, IN~47907, USA}
\email{pandey94@purdue.edu}

\author[Ramírez-Moreno]{Pedro Ramírez-Moreno}
\address{Departamento de Matemáticas, Universidad Autónoma Metropolitana, Unidad Iztapalapa, Mexico City, Mexico}
\email{pedro.ramirez@cimat.mx}

\author[Schwede]{Karl Schwede}
\address{Department of Mathematics, University of Utah, Salt Lake City, UT, USA}
\email{schwede@math.utah.edu}

\author[Sridhar]{Prashanth Sridhar}
\address{Department of Mathematics, University of Alabama, Tuscaloosa, AL, USA}
\email{psridhar1@ua.edu}

\thanks{Benozzo was supported by the European Union’s Horizon 2020 research and innovation programme under the Marie Skłodowska-Curie grant agreement No 101034255. Pandey was supported by the AMS--Simons Travel Grant ASTG-23-284908. Ramírez-Moreno was supported by SECIHTI Grants CBF-2023-2024-224 and CF-2023-G-33. Schwede was supported by NSF Grants \#2101800, \#2501903 and by the Simons Foundation SFI-MPS-TSM-00013051.}
\maketitle
\tableofcontents

\section{Introduction}

The log canonical threshold over the complex numbers $k = \CC$ and $F$-pure threshold over a field $k$ of characteristic $p > 0$, provide subtle and important invariants of hypersurface singularities for $f \in k[x_1, \dots, x_n]$ \cite{KollarSingularitiesOfPairs,LazarsfeldPositivity2,TakagiWatanabeFPureThresh,MustataTakagiWatanabeFThresholdsAndBernsteinSato}. Interpolating between those two worlds is the mixed characteristic realm, and so it is natural to explore the singularities of hypersurfaces in
\[
    f \in \bZ[x_1, \dots, x_n].
\]
As this is a local study, it is harmless to replace $\bZ$ by the $p$-adic integers $\bZ_p$ and consider $f \in \bZ_p[x_1, \dots, x_n]$, or even $f \in \bZ_p\llbracket x_1, \dots, x_n\rrbracket$.  In this ring, from the point of view of singularities, $p$ behaves like a variable.  Hence while $f = y^2 + x^3$ and $f = x^3 + y^3 + z^3$ define singularities over fields, choices of $f$ like 
\[
    f = p^2 + x^3 \;\;\; \text{ or } \;\;\; f = p^3 + x^3 + y^3
\]
yield singular hypersurfaces as well.  Building upon work and perspectives of \cite{CaiPandeQuinlanGallegoSchwedeTuckerpluspurethresholdscusplikesingularities,yoshikawa2025computationmethodperfectoidpurity,RodriguezVillalobos.BCMThresholdsOfHypersurfaces,MaSchwedeTuckerWaldronWitaszekAdjoint,MaSchwedeSingularitiesMixedCharBCM} and others, we study singularities of such hypersurfaces in mixed characteristic.  We now explain how precisely we measure these singularities.

Suppose $(R, \fm)$ is a complete regular Noetherian local ring of mixed characteristic and $0 \neq f \in \fm$.  We study the \emph{plus-pure threshold}\footnote{In our context, the plus pure threshold coincides with the \emph{BCM-threshold}  
of \cite{RodriguezVillalobos.BCMThresholdsOfHypersurfaces} with respect to the BCM-algebra $\widehat{R^+}$, and likewise essentially agrees with the \emph{BCM-regular threshold} of \cite[Examples 7.9, 7.10]{MaSchwedeTuckerWaldronWitaszekAdjoint}. It also appears as a jumping number of $+$-test ideals, see for instance \cite[Conjecture 8.4]{HaconLamarcheSchwede.GlobalGenTestIdeals}.  We believe it also coincides with a natural generalization of the \emph{perfectoid pure threshold} from   \cite{yoshikawa2025computationmethodperfectoidpurity}, see \autoref{rem.YoshikawaRemark}} of $f$ as coined in \cite{CaiPandeQuinlanGallegoSchwedeTuckerpluspurethresholdscusplikesingularities}.  Set $R^+$ to be the integral closure of $R$ in $\overline{K(R)}$, an algebraic closure of its field of fractions.  We can then define 
\[
    \ppt(f) \coloneqq \sup \{t \in \QQ_{>0} \;|\; R \xrightarrow{1 \mapsto f^t} R^+ \text{ splits}\}.
\]
Here $f^t$ makes sense in $R^+$ up to a unit, which does not affect splitting/purity\footnote{splitting and purity are equivalent here since $R$ is complete}.  As $R$ is regular, $\ppt(f)$ can also be characterized as 
\[
    \ppt(f) = \sup \{t \in \QQ_{>0}\;|\; f^t \notin \fm R^+\}.
\]
While relatively easy to define, the \emph{plus-pure threshold} seems to be very difficult to compute in mixed characteristic.  Even without resolution of singularities, one can define $\lct(f)$ based on all proper birational maps  and so it follows from \cite{BhattAbsoluteIntegralClosure,MaSchwedeSingularitiesMixedCharBCM} that 
\[
\ppt(f) \leq \lct(f) 
\]
quite generally.  At the same time, if $R = V\llbracket x_1, \dots, x_n\rrbracket$ where $(V, \varpi)$ is a mixed characteristic complete DVR, we also have 
\[
\ppt(f) \geq \fpt\big(\overline{f} \in V/(\varpi)\llbracket x_1, \dots, x_n \rrbracket\big).
\]
More precise comparisons can also be made with the $\fpt$ of the restriction of the strict transform of $V(f)$ to the exceptional divisor of a blow-up, or in other words by doing a computation on an associated graded ring (see \cite[Section 7]{MaSchwedeTuckerWaldronWitaszekAdjoint} and also compare with \cite[Section 7]{BMPSTWW1} and \cite{TakamatsuYoshikawaMMP}).

Using these observations as a starting point, in  \cite{CaiPandeQuinlanGallegoSchwedeTuckerpluspurethresholdscusplikesingularities} the authors studied $$\ppt(p^a + x^b \in \ZZ_p\llbracket x \rrbracket)$$ at least for certain values of $a$ and $b$, as well as other sporadic examples.  In \cite{yoshikawa2025computationmethodperfectoidpurity}, the author used quite different methods related to quasi-$F$-splittings to prove that certain equations like $x^3+y^3+z^3 \in \bZ_p\llbracket x,y,z \rrbracket$ define perfectoid pure singularities and hence have $\ppt = 1$, see \autoref{prop.PPTWhenHypersurfaceIsPerfdPure}.  

In this paper, building primarily on the methods of \cite{CaiPandeQuinlanGallegoSchwedeTuckerpluspurethresholdscusplikesingularities}, we study the plus-pure thresholds of various families of hypersurfaces.

Our first (relatively easy-to-prove) observation is a statement about the behavior of the $\ppt$ if one ramifies the base DVR.  

\begin{theoremA*}[{\autoref{cor.EqualityOnTheNoseAfterRamifying}, \autoref{cor.LimitingPPTStatement}}]
    Suppose $(V, \varpi)$ is a mixed characteristic $(0,p>0)$ complete DVR and $R = V\llbracket x_2, \dots, x_n\rrbracket$ has maximal ideal $\fm$.  Suppose $f \in \fm$ with corresponding $\overline{f} \in V/(\varpi)\llbracket x_2, \dots, x_n\rrbracket$.  Then 
    \[
        \lim_{e \to \infty} \ppt\big(f \in V[\varpi^{1/p^e}]\llbracket x_2, \dots, x_n \rrbracket\big) = \fpt\big(\overline{f} \in V/(\varpi)\llbracket x_2, \dots, x_n \rrbracket\big).
    \]
    Furthermore, if $\fpt(\overline{f} \in V/(\varpi)\llbracket x_2, \dots, x_n \rrbracket) = a/p^e$, then we have equality at the $e$-th stage of the limit:
    \[
         \ppt\big(f \in V[\varpi^{1/p^e}]\llbracket x_2, \dots, x_n \rrbracket\big) = a/p^e = \fpt\big(\overline{f} \in V/(\varpi)\llbracket x_2, \dots, x_n \rrbracket\big).
    \]
\end{theoremA*}

Applied to Yoshikawa's example, this immediately tells us that while $\ppt(x^3+y^3+z^3 \in \bZ_p\llbracket x,y,z\rrbracket) = 1$ for $p \equiv_3 2$, we have that 
\[
\ppt\big(x^3+y^3+z^3 \in \bZ_p[p^{1/p}]\llbracket x,y,z\rrbracket\big) = 1 - {1 \over p} = \fpt\big(x^3+y^3+z^3 \in \bF_p\llbracket x,y,z \rrbracket\big)
\]
for those same $p$, see \cite{BhattSinghThresholds,HernandezFInvariantsOfDiagonalHyp} as well as \cite{TakagiWatanabeFPureThresh,MustataTakagiWatanabeFThresholdsAndBernsteinSato}.

Using similar methods we can also compute plus-pure thresholds for certain Fermat-type hypersurfaces.  For example, let $R_a \coloneqq W(k)[p^{1/p^a}]\llbracket x_2,\dots, x_d\rrbracket$.
Let $f_a = p^{d/p^a} +x_2^d+ \dots + x_d^d \in R_a$, and let $f_0 \coloneqq x_1^d+ \dots +x_d^d \in k\llbracket x_1, \dots, x_d\rrbracket$.
    Fix $s \geq 1$ such that $p^s \leq d < p^{s+1}$, then we have that 
    \[
        \ppt(f_a) = \fpt(f_0)
    \] 
    for all $a \geq s$; see \autoref{l-homogeneoushyp_ramification} for the proof.
When $d = p^s+1$, $f_0$ is an example of an \emph{extremal singularity}.

Indeed, recently, there has been substantial interest in these so-called ``extremal hypersurface singularities'' in characteristic $p > 0$.  See \cite{KadyrsizovaKenkelPageSinghSmithVraciuWittLowerBoundsonFPureThresholdsandExtremalSingularities} as well as  \cite{Cheng.QBicThesis, Cheng.QBicForms, Cheng.QBicHypersurfaces,KadyrsizovaPageSinghSmithVraciuWitt.ClassificationOfFrobeniusForms, SmithVraciu.ValuesOfFPT}.  In general, if $f \in \bF_p[ x_1, \dots, x_n ]$ is homogeneous of degree $d$, then one always has the lower bound \cite[Theorem 3.1]{KadyrsizovaKenkelPageSinghSmithVraciuWittLowerBoundsonFPureThresholdsandExtremalSingularities} 
\[
    \fpt(f) \geq {1 \over d-1}.
\]
Furthermore, this bound can only be an equality if $d = p^e + 1$ for some integer $e > 0$ and when $f =  \sum_{i=1}^n x_i^{p^e}L_i$, for $L_i$ linear forms. Such a polynomial  $f$ with minimal $F$-pure threshold is said to have an \emph{extremal singularity} in characteristic $p$. We consider ``extremal-looking singularities" in mixed characteristic. By generalizing the arguments of \cite[Lemmas 4.2, 4.3]{CaiPandeQuinlanGallegoSchwedeTuckerpluspurethresholdscusplikesingularities}, it turns out that we \textit{cannot} obtain the analogous ``extremal'' singularities in unramified regular rings of mixed characteristic:

\begin{theoremB*}[Extremal singularities, \autoref{cor.EasyExtremalBound}]
Let $k$ be a perfect field of characteristic $p>0$ and fix $n \geq 3$.
   Let $$f = p^{p^{e} + 1} + x_2^{p^e + 1} + \dots + x_n^{p^e + 1} \; \text{ or } \; f = p^{p^{e} + 1} + x_2^{p^e}x_3 + x_3^{p^e}x_2 + x_4^{p^e + 1} + \dots + x_n^{p^e + 1} \in W(k)\llbracket x_2, \dots, x_n\rrbracket.$$ In either case, we have that the modulo $p$ reduction, $\overline{f} \in \bF_p \llbracket x_2, \dots, x_n\rrbracket$, is an extremal singularity. However,  
   \[
    \ppt(f) > \fpt(\ov{f}) = \frac{1}{p^e}.
   \]
\end{theoremB*}
For much more general results, see \autoref{thm:lowerbound}.
Note that in the context of Theorem B, if we adjoin $p^e$-th roots of $p$, we do obtain $\ppt(f) = {1 \over p^e}$ by an application of Theorem A or even replacing the term $p^{p^e+1}$ by $p^{p^e+1 \over p^{e}}$, as in \autoref{l-homogeneoushyp_ramification}.

We also explore choices of $f$ related to supersingular elliptic curves.   We studied $x^3+y^3+z^3 \in V\llbracket x,y,z\rrbracket$ above, but it is also natural to consider 
\[
    \ppt\big(p^3+x^3+y^3 \in W(k)\llbracket x,y\rrbracket\big),
\]
as suggested at the end of \cite{CaiPandeQuinlanGallegoSchwedeTuckerpluspurethresholdscusplikesingularities}.
While we have been unable to compute its $\ppt$ in general, we do obtain the following striking bounds:

\begin{theoremC*}[Elliptic curves, \autoref{thm:ppt_elliptic}]
   Let $k$ be a perfect field of characteristic $p$ with $p \equiv_3 2$. Consider $f = p^3 + x^3+y^3 \in W(k)\llbracket x,y\rrbracket$.  Then 
    \[
    \ppt(f)\in \left[1-\frac{1}{p}, 1-\frac{1}{p^2} \right]. 
    \]
    Furthermore, in characteristic $2$, we have a strict inequality on the left via Theorem B: 
    \[
        \ppt(f) \in \left(\frac{1}{2},\frac{3}{4} \right].
    \]
\end{theoremC*}
These bounds apply to other similar equations as well, see \autoref{thm:lowerbound} and \autoref{thm:ppt_elliptic} for more general statements.

In mixed characteristic $(0,p=2)$, the example above also yields an interesting observation about jumping numbers; let us begin with some background:  Given $f \in R$ in equal characteristic $p > 0$, recall that $\lambda > 0$ is an $F$-jumping number if 

$$\tau(R, f^{\lambda - \epsilon}) \neq \tau(R, f^{\lambda + \epsilon}) \; \text{ for all }\; 0<\epsilon \ll 1;$$ here $\tau$ denotes the test ideal of \cite{HaraYoshidaGeneralizationOfTightClosure}.  If $R$ is regular, the smallest jumping number is the $F$-pure threshold.   Similarly, if $R$ is complete regular local and $f \in R$, then $\ppt\big(f)$ is also the first jumping number of the BCM test ideal $\tau_{+}(f^{\lambda})$ of \cite{MaSchwedeSingularitiesMixedCharBCM} computed with respect to the perfectoid BCM-algebra $\widehat{R^+}$, see also \cite{BhattAbsoluteIntegralClosure,RodriguezVillalobos.BCMThresholdsOfHypersurfaces}.  It is then natural to ask if properties of $F$-jumping numbers also hold in mixed characteristic.  For example, in 
\cite[Lemma 3.1(1)]{BlickleMustataSmithDiscretenessAndRationalityOfFThresholds}, it is shown that if $\lambda$ is an $F$-jumping number in characteristic $p > 0$, then so is $p\lambda$ and hence so is the fractional part $\{ p \lambda \}$.

Theorem C implies that the corresponding statement is \textit{false} in mixed characteristic regular rings.

\begin{observation*}
$p \cdot \ppt(f)$, and hence $\{p \cdot \ppt(f)\}$, is not always a jumping number of the associated $+$-test ideal in mixed characteristic.
\end{observation*}
\noindent
For more discussion, see \autoref{rem.PTimesJumpingNumber}.

Next, we explore equations $f$ whose ramification is controlled when adjoining $f^{1/p}$; see \autoref{thm:intermediate} and \autoref{thm:2} for the precise and most general statements. One consequence of these results is that given a complete unramified regular local ring of mixed characteristic $p>0$, any ball of radius $1/p^2$ (in the p-adic metric) centered on a $p$-th power consists of non-perfectoid pure forms, see \autoref{prop.PPTWhenHypersurfaceIsPerfdPure} and \autoref{rem:URLR}. 
Explicitly, for common base rings, it looks like:

\begin{theoremD*}[{\autoref{thm:intermediate}, \autoref{thm:2}}]
Let $\zeta$ denote a primitive $p$-th root of unity. 
\begin{enumerate}
    \item For any $f\in \ZZ_p[\zeta]\llbracket x_2,\dots,x_d\rrbracket$ admitting a $p$-th root modulo $(\zeta-1)^p$, we have $\ppt(f)\leq 1/p$.

    \item For any $f\in \ZZ_p\llbracket x_2,\dots,x_d\rrbracket$ admitting a $p$-th root modulo $p^2$, we have $\ppt(f)\leq 1-1/p$.
\end{enumerate}
   
\end{theoremD*}

In the setting of Theorem D (a), if $f$ admits a linear $p$-th root modulo $(\zeta-1)^p$ (for example, $p=2$ and $f=(x_2+\dots+x_d)^2 + 4a$ for some $a\in R$), then $\ppt(f)=1/p$ $(=1/2$). This follows by combining the upper bound of Theorem D (a) with the lower bound coming from the mod $p$ reduction.

Finally, we study some hypersurface singularities whose mod $p$ reduction is not reduced. In \cite[Proposition 4.6]{CaiPandeQuinlanGallegoSchwedeTuckerpluspurethresholdscusplikesingularities}, the authors showed that $$\ppt(x^2+2^2 \in \ZZ_2 \llbracket x\rrbracket)=\fpt(x^2+y^2 \in \FF_2\llbracket x,y\rrbracket)=1/2.$$
The following result shows that the analogous statement does \textit{not} hold for any power of an odd prime and in any dimension. In particular, this partially answers \cite[Question 5.1]{CaiPandeQuinlanGallegoSchwedeTuckerpluspurethresholdscusplikesingularities}. 

\begin{theoremE*}[\autoref{TuckersBane}]
    Let $k$ be a perfect field of characteristic $p>2$. Let $f = p^{p^e}+ x_2^{p^e} + \dots + x_n^{p^e} \in W(k)\llbracket x_2, \dots, x_n\rrbracket $ and $f_0=x_1^{p^e} + \dots + x_n^{p^e} \in k\llbracket x_1, \ldots, x_n\rrbracket $. Then $$\ppt(f) > \fpt(f_0) = \frac{1}{p^e}.$$ 
\end{theoremE*}

\subsection*{Acknowledgements}

The authors began this project at the Fields Institute in Toronto as part of the Apprenticeship Program in Commutative Algebra in January 2025.  We appreciate the Fields Institute's support.
The authors thank Anna Brosowsky, Hanlin Cai, Linquan Ma, Eamon Quinlan-Gallego, Kevin Tucker, and Shou Yoshikawa for valuable conversations.  We also thank Linquan Ma and the anonymous referee for comments on previous drafts.

\section{Preliminaries}

Throughout, if $R$ is an integral domain, then $R^+$ is an absolute integral closure. By $p$, we will always denote a positive prime integer.  If we write $a \equiv_p b$ then we mean that $a$ and $b$ are equivalent modulo $p$.  More generally, for any ideal $I$, we write $a \equiv_I b$ when $a + I = b+I$.  

\begin{definition}
    Suppose $(R, \fm)$ is a complete Noetherian local domain of mixed characteristic $(0, p> 0)$.  Suppose $0 \neq f \in \fm$.  We define the \emph{plus-pure threshold} to be
    \[
        \ppt(f \in R) \coloneqq \sup \{t \in \QQ_{>0} \;|\; R \xrightarrow{1 \mapsto f^t}\widehat{R^+} \text{ is pure}\}.
    \]
    Note that $f^t$ is only defined up to units in $\widehat{R^+}$, but a unit will not change whether the map $R \xrightarrow{1 \mapsto f^t}\widehat{R^+}$ is pure.
    When $R$ is clear from the context, we simply write $\ppt(f)$.
\end{definition}

\begin{remark}
\label{rem.YoshikawaRemark}
    The notation $\ppt(-)$ was used in \cite{yoshikawa2025computationmethodperfectoidpurity} in the special case $f = p$ for the related notion of the \emph{perfectoid pure threshold}.  Based on the equal characteristic $p > 0$ picture, we expect the perfectoid pure threshold to agree with the plus-pure threshold, at least in a regular ambient ring, the context of this paper, see also \cite[Remark 2.3]{CaiPandeQuinlanGallegoSchwedeTuckerpluspurethresholdscusplikesingularities}.  Because of this, we do not anticipate confusion.
\end{remark}

If $R$ is regular, then we have the alternate description:
\begin{lemma}[{\cite[Definition 2.1]{CaiPandeQuinlanGallegoSchwedeTuckerpluspurethresholdscusplikesingularities}}]
\label{lem.AlternateDescriptionPPT}
With notation as above:
    \[
        \begin{array}{rl}
            \ppt(f) = & \sup\{t \in \QQ_{>0}\;|\; f^t \notin \fm \widehat{R^+} \}\\
                = & \inf\{t \in \QQ_{>0}\;|\; f^t \in \fm {R^+} \}.
        \end{array}
    \]
\end{lemma}

One should compare this with the definition of the $F$-pure threshold.
\begin{definition}
\label{def.fpt}
    Suppose that $(R, \fm)$ is a complete regular local ring of positive characteristic $p > 0$ and $f \in \fm$.  Then $\fpt(f \in R) = \sup\{ {\nu \over p^e} \;|\; f^\nu \notin \fm^{[p^e]} \} = \sup\{t \in \QQ_{>0}\;|\; f^t \not\in \fm R^+ \}.$
    
    When $R$ is clear from the context, we simply write $\fpt(f)$.
\end{definition}
The second equality in the definition above is well known. For a generalization, see for instance \cite[Proposition 2.0.4]{RodriguezVillalobos.BCMThresholdsOfHypersurfaces}. Though the equality is stated for $F$-finite strongly $F$-regular rings, the proof generalizes to the complete regular local setting above as such rings are still strongly $F$-regular in the classical sense. 

\begin{lemma}
    \label{lem.EtaleExtensionPPT}
    Suppose $(R,\fm) \subseteq (S,\fn)$ is a finite  extension of complete regular local rings of mixed characteristic $(0, p> 0)$ such that $\fm S = \fn$ (for instance if the extension is \'etale).  Suppose $f \in \fm$.  Then 
    \[
        \ppt(f \in R) = \ppt(f \in S).
    \]
\end{lemma}
\begin{proof}
    As $R \subseteq S$ is finite, we see that $R^+ = S^+$. The result follows from \autoref{lem.AlternateDescriptionPPT}.  
\end{proof}

We note the following comparison between the $\fpt$ and $\ppt$ which is implicit in \cite{CaiPandeQuinlanGallegoSchwedeTuckerpluspurethresholdscusplikesingularities}.
\begin{lemma}
\label{lem.pptVsFPTEasyComparison}
    Suppose that $(R, \fm)$ is a complete regular local ring of mixed characteristic $(0, p>0)$ and $0 \neq \varpi \in \fm$ is such that $R/(\varpi)$ is regular of characteristic $p > 0$ (and hence $\varpi | p$). Fix $f \in \fm$ with corresponding $\overline{f} \in R/(\varpi)$.  Then 
    \[
        \ppt(f) \geq \fpt(\overline{f}).
    \]
\end{lemma}
\begin{proof}
    Suppose $f^t \in \fm R^+$.  Then $\overline{f^t} \in \fm (R^+/(\varpi))$. Given a surjection $R \onto R/\varpi$ one obtains a map $R^+ \to (R/\varpi)^+$, see for instance \cite[Proposition 1.2]{HochsterHunekeApplicationsofBigCM}, which certainly sends $\varpi \mapsto 0$ and so factors through $R^+/(\varpi)$.
    Therefore $\overline{f^t}$ maps to some choice of ${\overline{f}}^t$ since $(R/(\varpi))^+$ is an integral domain. We get that ${\overline{f}}^t \in \fm (R/(\varpi))^+$ and the result follows.
\end{proof}

We will repeatedly use the following lemma.
\begin{lemma}[ {\cite[Lemma 2.2]{CaiPandeQuinlanGallegoSchwedeTuckerpluspurethresholdscusplikesingularities}}] \label{l-lemma2.3}
Let $S$ be an absolutely integrally closed domain, let $z, y, y_1,\dots,y_s \in S$ be elements, let $e \geq 1$ be an integer and $\epsilon \in \QQ$ be a rational number.  Finally suppose that $\varpi\in S$ divides a prime $p > 0$.
\begin{enumerate}
    \item[(i)] If $\epsilon \in \left( 0, \frac{p}{p-1}\right]$, we have \label{l-lemma2.3.i}
    \[
    z \in (\varpi^{\epsilon}, y) \iff z^{1/p^e} \in (\varpi^{\epsilon/p^e}, y^{1/p^e}).
    \]
    \item[(ii)] If $\epsilon \in \left( 0, 1 \right]$, we have \label{l-lemma2.3.ii}
    \[
    z \in (\varpi^{\epsilon}, y_1, \dots, y_s) \iff z^{1/p^e} \in (\varpi^{\epsilon/p^e}, y_1^{1/p^e}, \dots, y_s^{1/p^e}).
    \]
\end{enumerate}
\end{lemma}
\begin{proof}
    This was shown in \cite{CaiPandeQuinlanGallegoSchwedeTuckerpluspurethresholdscusplikesingularities} in the case that $p = \varpi$.  The proof works verbatim the same as soon as one notices that $\varpi^{\epsilon} | p^{\epsilon}$  and so we do not reproduce it here.
\end{proof}

\subsection{Plus-pure threshold vs perfectoid pure hypersurfaces}

The following result is well known to experts but we do not know a reference.  In the proof we use the following terminology.  Suppose $B$ is a perfectoid $A$-algebra and $I \subseteq A$ is n ideal.   We recall that the perfectoidization $(IB)_{\perfd}$ of $IB$ is precisely the kernel of $B \to (B/IB)_{\perfd}$ where this $(-)_{\perfd}$ denotes perfectoidization of a $B$-algebra, see \cite{BhattScholzepPrismaticCohomology} and \cite[Section 2.4]{CaiLeeMaSchwedeTucker}.   We observe that if $I = (f)$ and if $B$ has a compatible system of $p$th power roots of $f$, then $(f B)_{\perfd} = (f^{1/p^{\infty}})^{-}$ where by ${ }^{-}$ we mean $p$-adic closure, see \cite[Lemma 2.4.3]{CaiLeeMaSchwedeTucker}.  Finally, note that if $f = g^N$ for some integer $N$, then $(fB)_{\perfd} = (g B)_{\perfd}$ as perfectoid rings are reduced and so $g$ already maps to zero in $B/(fB)_{\perfd}$.

\begin{proposition}
\label{prop.PPTWhenHypersurfaceIsPerfdPure}
    Suppose $A = W(k)\llbracket x_2, \dots, x_n\rrbracket$ with $k$ perfect and let $\fm$ be the maximal ideal with some $0 \neq f \in \fm$.  Then $A/(f)$ is perfectoid pure if and only if $\ppt(f) = 1$.
\end{proposition} 
Related statements are true in greater generality, and even for ramified $A$; details can be found in \autoref{rem.WeakenHypothesesOfPPT} below.
\begin{proof}
    Suppose first that $\ppt(f) = 1$.  Then the map $A \xrightarrow{1 \mapsto f^{1-\epsilon}}\widehat{A^+}$ is pure for every $1 \gg \epsilon > 0$.  This is equivalent to the purity of 
    \[
        A \xrightarrow{1 \mapsto 1} f^{-(1-\epsilon)} \widehat{A^+}.
    \]
    But that is equivalent to the purity of the inclusion
    \[
        A \xrightarrow{1 \mapsto f}  f^{\epsilon} \widehat{A^+}.
    \]
    for all $1 \gg \epsilon > 0$.  Let $(f \widehat{A^+})_{\perfd}$ denote the perfectoidization of the ideal $f \widehat{A^+}$. As $\widehat{A^+}$ has a compatible system of $p$-power roots of $f$, $(f \widehat{A^+})_{\perfd} = (f^{1/p^\infty} \widehat{A^+})^{-}$ where $(\bullet)^{-}$ denotes $p$-closure.  As purity can be checked by verifying the injectivity of the map after tensoring with $E$, the injective hull of the residue field of $A$ (whose elements are $p$-power torsion), we see that the $p$-closure is harmless and hence 
    \[
        A \xrightarrow{1 \mapsto f}  (f\widehat{A^+})_{\perfd}
    \]
    is also pure.  We then see from \cite[Proposition 6.5]{BMPSTWW3} that $A/(f)$ is perfectoid pure as $A$ is Gorenstein so that perfectoid injectivity and perfectoid purity coincide.

        Conversely, suppose that $A/(f)$ is perfectoid pure.  The main idea is that  by  \cite[Theorem 6.6]{BMPSTWW3}, we know there exists a perfectoid pure $A$-algebra $B$ such that $(f) \to (fB)_{\perfd}$ is split.  We will show that this implies that $A \xrightarrow{1 \mapsto f^{1-\epsilon}} \widehat{A^+}$ is pure.
        
        Consider the auxiliary ring $R := A[T]/(T^{N}-f)$ for some integer $N$ not divisible by $p$ or $2$, we first observe that $R$ is an integral domain.  Since $A[T]$ is regular, it suffices to show that $S = K(A)[T]/(T^{N}-f)$ is an integral domain.  If $S$ is not an integral domain, then for instance by \cite[Chapter VI, Theorem 9.1]{Lange.Algebra}, we have that $f$ is an $\ell$-th power for some prime $\ell$ dividing $N$.  
        If $f = h^\ell$ is a $\ell$th power, then $h \in R$ by normality, and so $A/(f)=A/(h^\ell)$ is not reduced and hence is not perfectoid pure, contradicting our assumption.   
        Thus we may assume that $R$ is an integral domain.

    Set $A_{\infty} \coloneqq A[p^{1/p^{\infty}}, \dots, x_n^{1/p^{\infty}}]^{\wedge_p}$, 
    and let $R_{\infty} \coloneqq (R \otimes_A A_{\infty})_{\perfd}$.  
    As mentioned above, by \cite[Theorem 6.6]{BMPSTWW3}, there exists a perfectoid $B$ such that $(fA) \to (fB)_{\perfd}$ is pure.  
    By Andr\'e's flatness lemma \cite[Theorem 7.14]{BhattScholzepPrismaticCohomology} and the purity assumption (see for instance also \cite[Lemma 4.5]{BMPSTWW3}), 
    we can suppose that $p, x_2, \dots, x_n$ have compatible systems of $p$th power roots in $B$ and hence we can choose $B$ to be a perfectoid $A_{\infty}$-algebra 
    (this is the same argument as \cite[Lemma 4.23]{BMPSTWW3}).  
    Furthermore, as above we can assume that $f$ has an $N$th root $g$, and that $g$ has a compatible system of $p$th power roots, and hence $f$ also has such a system (note, we can even choose $B$ absolutely integrally closed). 
    By construction and universal properties, we have a map $R_{\infty} := (R \otimes_A A_{\infty})_{\perfd} \to B$ sending $T \mapsto g$.  Thus, from here on, we also write $g = T \in R$.  We notice that in a map $R_{\infty} \to R^+$, that $g$ must be sent to some $f^{1/N}$, an $N$th root of $f$.

    As $g^N = f$ and since $(f) \to (fB)_{\perfd} = (g B)_{\perfd} = (g^{1/p^\infty})^-$ is pure, we see that $A \xrightarrow{1 \mapsto g^{(N-1)} g^{1/p^{e}}} B$ is also pure for $e \gg 0$.      Set $\eta$ to be a socle generator for $E$, the injective hull of the residue field of $A$.  Then
    \[
        (g^{1/(p^\infty)}) \big( g^{N-1} \otimes \eta \big) \subseteq B \otimes_A E
    \]
    is nonzero for $e \gg 0$. Hence, as elements of $E$ are annihilated by powers of $p$, we also have that 
    \[
        (f R_{\infty})_{\perfd} \big( g^{N-1} \otimes \eta \big) \subseteq R_{\infty} \otimes_A E = R_{\infty} \otimes_R (R \otimes_A E)
    \]
    is nonzero.  Finally, by \cite[Lemma 5.1.6]{CaiLeeMaSchwedeTucker}, using that $A \to R$ is \'etale outside of $V(f)$, and since $g \mapsto f^{1/N} \in R^+$, we see that 
    \[
        (f^{1/p^{\infty}}) \big( f^{(N-1)/N} \otimes \eta \big) \subseteq \widehat{R^+} \otimes_R (R \otimes_A E) = \widehat{A^+} \otimes_A E
    \]
    is also nonzero.
    
    Since this holds for every such integer $N$, we see that the map 
    \[
        A \xrightarrow{f^{1-\epsilon}} \widehat{A^+}
    \]
    is pure for any rational $1 \gg \epsilon > 0$.  This completes the proof.
\end{proof}

\begin{remark}
\label{rem.WeakenHypothesesOfPPT}
    In an arbitrary Cohen-Macaulay local domain $A$ (not necessarily regular), if $\ppt(f) = 1$ then the argument above shows that $A/fA$ is perfectoid injective (and hence is perfectoid pure if $A$ is additionally Gorenstein) without substantial change.

    The converse argument used a base perfectoid ring $A_{\infty} = W(k)\llbracket x_2, \dots, x_d\rrbracket[p^{1/p^{\infty}}, \dots, x_d^{1/p^{\infty}}]^{\wedge_p}$ as in \cite[Lemma 5.1.6]{CaiLeeMaSchwedeTucker}.  If the original ring for instance is $A = W(k)\llbracket p^{1/p^n}, x_2, \dots, x_d\rrbracket$, which also embeds in $A_{\infty}$, then our argument, as well as that of \cite[Lemma 5.1.6]{CaiLeeMaSchwedeTucker}, doesn't change and the same conclusion holds.

    Based on the characteristic $p > 0$ picture, both directions of \autoref{prop.PPTWhenHypersurfaceIsPerfdPure} should also hold in any +-regular/splinter ambient ring (and the analogous result should hold for the \emph{perfectoid-pure threshold} in a perfectoid pure ring).  We do not attempt this however as we do not need it.
\end{remark}

\subsection{Diagonal hypersurfaces}

Let us recall a result of Hern\'andez about $F$-pure thresholds of diagonal hypersurfaces.
\begin{theorem}[{\cite[Theorem 3.4 and Corollary 3.9]{HernandezFInvariantsOfDiagonalHyp}}] \label{t-HerDiagHyp}
  Let $f\coloneqq x_1^{s_1}+x_2^{s_2}+ \dots+x_n^{s_n} \in k[x_1, \dots, x_n]$, where $k$ is a perfect field of characteristic $p>0$.
Write $\frac{1}{s_i} = \sum_{e \geq 1} p^{-e}s_{i,e}$, so that $s_{i,e}$ are not eventually zero and $0 \leq s_{i,e} \leq p-1$.
Define $L\coloneqq \min \{ e\geq 0 : \sum_{i=1}^n s_{i,e+1} \geq p \}$ and assume $L < \infty$.
Then
\[
    \fpt(f) = \frac{1}{p^L} \left( \sum_{e=1}^L \sum_{i=1}^n p^{L-e}s_{i,e} +1 \right).
    \]
If, instead, $L=\infty$, then
\[
\fpt(f)= \sum_{i=1}^n 1/s_i.
\]

In particular, if $f = \sum_{i=1}^{d} x_i^d \in k[x_1, \dots, x_{d}]$ for some $d > 0$, then
    $$\fpt(f) = \begin{cases}
        \frac{1}{p^s} & \qquad \textrm{if }p^s \leq d < p^{s+1} \textrm{ for some }s \geq 1; \\
        1 - \frac{\alpha - 1}{p} & \qquad \textrm{if } 0 < d < p \textrm{ and } p \equiv_d \alpha, 1 \leq \alpha < d.
        \end{cases}$$
\end{theorem}

We prove that for diagonal hypersurfaces, the $F$-pure threshold of the corresponding equation in positive characteristic is always a lower-bound, with a blow-up argument similar to \cite[Remark 2.10 and Proposition 2.8]{CaiPandeQuinlanGallegoSchwedeTuckerpluspurethresholdscusplikesingularities}.

\begin{lemma} \label{l-blow-up}
Let $(V,\varpi)$ be a mixed characteristic $(0,p>0)$ complete DVR with uniformizer $\varpi$. Suppose
\begin{itemize}
    \item[(i)] either that $R=V\llbracket x \rrbracket$, $f\coloneqq \varpi^a+x^b \in R$ and $f_0\coloneqq y^a+x^b \in V/(\varpi)\llbracket x, y \rrbracket$;
    \item[(ii)] or that $R = V\llbracket x_2, \dots, x_n\rrbracket$, $f\coloneqq \varpi^d + x_2^d+\dots+x_n^d \in V\llbracket x_2, \dots, x_n \rrbracket$ and $f_0 \coloneqq x_1^d + x_2^d+\dots+x_n^d \in V/(\varpi)\llbracket x_2, \dots, x_n \rrbracket$.
\end{itemize}
    Then $\fpt(f_0) \leq \ppt(f) \leq \lct(f)$.
\end{lemma}

\begin{proof}
Case (i) follows from \cite[Proposition 2.8]{CaiPandeQuinlanGallegoSchwedeTuckerpluspurethresholdscusplikesingularities}.\footnote{There it is stated for the ring $\ZZ_p \llbracket x \rrbracket$, but the same argument works for $R$.}
As for case (ii), let $\pi\colon X \to \Spec(R)$ be the blow-up at the origin $V(\varpi, x_2, \dots, x_n)$ and let $E \simeq \mathbb{P}^{n-1}$ be the exceptional divisor. Note that the strict transform of $\mathrm{div}(f)$ in $E$ is defined by $f_0$.
The discrepancy over $(\Spec(R), t\mathrm{div}(f))$ is $n-1-td$.
Since $+$-regular singularities are in particular klt, we have $\ppt(f) \leq n/d= \lct(f)$. 
If $t < \fpt(f_0)$, then $(E, t\mathrm{div}(f_0))$ is globally $F$-regular by \cite[Proposition 5.3]{SchwedeSmithLogFanoVsGloballyFRegular}, therefore, by \cite[Lemma 7.2]{MaSchwedeTuckerWaldronWitaszekAdjoint}, $(\Spec(R), t\mathrm{div}(f))$ is $+$-regular, showing that $\ppt(f) \geq \fpt(f_0)$.
\end{proof}

\subsection{Combinatorial inputs}
We discuss some combinatorial identities needed for our results. To begin with, we recall a classical result of Kummer, used to compute the $p$-adic valuation $v_p\binom{n}{m}$ of the binomial coefficient $\binom{n}{m}$:
\begin{remark}[{Kummer's  Theorem \cite{KummersThm}}] \label{r-Kummer}
Let $p$ be a positive prime integer.
Write the base-$p$ expansion of a natural number $n$ as $n = n_rp^r+n_{r-1}p^{r-1}+\ldots+n_1p+n_0$, with $n_i \in \{ 0, \dots, p-1\}$, and denote $S_p(n) \coloneqq n_0+n_1+\ldots +n_r$. Then
\[v_p\binom{n}{m} = \frac{S_p(m)+S_p(n-m)-S_p(n)}{p-1}.\]
\end{remark}

\begin{lemma}\label{lem: binomialValuation}
Let $p$ be a positive prime integer. Let $e$ and $i$ be positive integers such that $1 \leq i \leq p^e$. Then
$$ v_p\binom{p^e}{i} =  e - v_p(i).$$
\end{lemma}

\begin{proof}
To prove the equality, we show that the sum $v_p\binom{p^e}{i} + v_p(i)$ is equal to $e$. Observe that $v_p(i) =  v_p\binom{i}{1}$. Applying Kummer's Theorem \autoref{r-Kummer} to each of the $p$-adic valuations, we obtain:
$$v_p\binom{p^e}{i} = \frac{S_p(i) + S_p(p^e - i) - S_p(p^e)}{p-1} = \frac{S_p(i) + S_p(p^e - i) - 1}{p-1} $$
and
$$v_p\binom{i}{1} = \frac{S_p(1) + S_p(i - 1) - S_p(i)}{p-1} = \frac{1 + S_p(i - 1) - S_p(i)}{p-1}. $$
Hence, 
$$v_p\binom{p^e}{i} + v_p\binom{i}{1} = \frac{S_p(p^e - i) + S_p(i - 1)}{p-1}.$$

 Now we proceed to prove that $S_p(p^e - i) + S_p(i - 1) = e(p - 1)$ to conclude the proof. Notice that $p^e - i$ and $i - 1$ are non-negative integers such that their sum is equal to $p^e - 1$. Since the base-$p$ expansion of $p^e - 1$ is 
 $$ (p-1) \cdot p^{e-1} + (p-1) \cdot p^{e - 2} + \dots + (p-1) \cdot p^1 + (p-1) \cdot p^0,$$
 then the base-$p$ expansions of $p^e - i$ and $i - 1$ are of the form
 $$ a_{e-1} \cdot p^{e-1} + a_{e-2} \cdot p^{e - 2} + \dots + a_1\cdot p^1 + a_0 \cdot p^0$$
 and
 $$ b_{e-1} \cdot p^{e-1} + b_{e-2} \cdot p^{e - 2} + \dots + b_1\cdot p^1 + b_0 \cdot p^0$$
 respectively, where $a_j + b_j = p - 1$ for every $j$. This implies that 
 $$S_p(p^e - i) + S_p(i - 1) = S_p(p^e - 1) = e(p - 1),$$
 which concludes the proof.
\end{proof}

\begin{lemma} \label{l-magic}
For $p\equiv_3 2$,
\[
\frac{2p^2-2}{3} = \frac{2p-1}{3}p +\frac{p-2}{3}
\]
and
\[
\frac{p^2-1}{3} = \frac{p-2}{3}p +\frac{2p-1}{3}
\]
are the base-$p$ expansions for $\frac{2p^2-2}{3}$ and $\frac{p^2-1}{3}$ respectively. 
\end{lemma}
\begin{proof}
    The equalities above are clearly true algebraically, and as $p \equiv_3 2$, both $\frac{2p-1}{3}$ and $\frac{p-2}{3}$ are integers strictly less than $p$, and hence not $p$-divisible. Thus we conclude these are base-$p$ expansions.
\end{proof}

\begin{lemma}
\label{lem: LucasThmCor}
    For $p > 3$ a positive prime integer such that $p \equiv_3 2$, set $k = \frac{p^2 - 1}{3}$. Then $k$ is an integer and $p$ divides $\binom{2k}{k}$. 
\end{lemma}

\begin{proof}
    If $p > 3$ then $p$ is equivalent to $1$ or $2$ mod $3$, so $p^2 \equiv_3 1$. From this we easily see that $k \in \ZZ$. We then use Lucas's Theorem \cite{LucasThm} and \autoref{l-magic} to conclude that
    $$\binom{2k}{k} \equiv_p \binom{(2p-1)/3}{(p-2)/3} \cdot \binom{(p-2)/3}{(2p-1)/3}.$$
    Since $\frac{p-2}{3} < \frac{2p-1}{3}$ for any $p > 3$, it follows that $\binom{(p-2)/3}{(2p-1)/3} = 0$ and thus $\binom{2k}{k} \equiv_p 0$. 
\end{proof}

\section{Ramification and plus-pure thresholds}\label{sec:ramification_ppt}

The point of this section is to make some observations on the connection between ramification over $p$ in finite extensions and plus-pure threshold of hypersurfaces.

\begin{lemma}
\label{lem.EasyLemmaBoundWhenAddingRoots}
    Suppose $R = V\llbracket x_2, \dots, x_n\rrbracket$ where $(V,\varpi)$ is a mixed characteristic $(0,p>0)$ complete DVR with uniformizer $\varpi$.  Let $0 \neq f \in R$ and let $\overline{f}$ denote the image of $f$ in $R/(\varpi) = V/(\varpi)\llbracket x_2, \dots, x_n\rrbracket$.  Suppose that $\fpt(\overline{f}) \leq {a/p^e}$ and that $V' \supseteq V$ is an extension of DVRs, where $V'$ contains some $p^e$-th root of $\varpi$, which we denote by $\varpi^{1/p^e}$.  Then $\ppt(f \in V'\llbracket x_2, \dots, x_n\rrbracket) \leq {a/p^e}$ as well.
\end{lemma}
\begin{proof}    
Note that $p \in \varpi R$ and hence we also have that $p^{1/p^e} \in \varpi^{1/p^e} R^+$ for any choice of $p^e$-th roots. 
It suffices to show that if $\fpt(\overline{f}) < a/p^e$, then $\ppt(f \in V'\llbracket x_2, \dots, x_n\rrbracket) < {a/p^e}$ so let us assume the former.  By assumption $\overline{f} \in (x_2^{p^e}, \dots, x_n^{p^e})$, and hence $f^a \in (x_2^{p^e}, \dots, x_n^{p^e}, \varpi)$.  But now applying \autoref{l-lemma2.3}(ii), we see that $f^{a/p^e} \in (x_2, \dots, x_n, \varpi^{1/p^e})R^+$, and the result follows.
\end{proof}

\begin{corollary}
\label{cor.EqualityOnTheNoseAfterRamifying}
With notation as in \autoref{lem.EasyLemmaBoundWhenAddingRoots}, suppose $\fpt(\overline{f}) = a/p^e$ for some integer $a$ (that is, the base-$p$ expansion of $\fpt(\overline{f})$ terminates after $e$ steps).  Then $\ppt(f \in V'\llbracket x_2, \dots, x_n\rrbracket) = \fpt(\overline{f})$.
\end{corollary}
\begin{proof}
    We know that $\ppt(f) \geq \fpt(\overline{f})$ by \autoref{lem.pptVsFPTEasyComparison}.  Now apply \autoref{lem.EasyLemmaBoundWhenAddingRoots} for the reverse inequality.
\end{proof}

Over a perfect field of characteristic $p > 0$, as any $g = f(x_1^p, \dots, x_n^p)$ is a $p$-th power, we see that $\fpt(g) \leq 1/p$.  The same holds in mixed characteristic if we also extract the $p$-th root of $p$.

\begin{cor}
\label{cor.PthPowersRamified}
    {With notation as in \autoref{lem.EasyLemmaBoundWhenAddingRoots}, assume the residue field of $V$ is perfect. Then for any $f\in V\llbracket x_2,\dots,x_n \rrbracket$, we have $\ppt(f\in V[\varpi^{1/p^e}]\llbracket x_2^{1/p^e},\dots,x_n^{1/p^e}\rrbracket)\leq 1/p^e$.}
\end{cor} 

\begin{proof}
    {Any $f\in V\llbracket x_2,\dots,x_n \rrbracket$ admits a $p^e$-th root modulo $\varpi$ in $V\llbracket x_2^{1/p^e},\dots,x_n^{1/p^e}\rrbracket$ so that 
    $$\fpt
\left(\overline{f}\in V/(\varpi)\llbracket x_2^{1/p^e},\dots,x_n^{1/p^e}\rrbracket\right)\leq 1/p^e.$$
    Now apply \autoref{lem.EasyLemmaBoundWhenAddingRoots} to $f\in V\llbracket x_2^{1/p^e},\dots,x_n^{1/p^e}\rrbracket$.}
\end{proof}

Combining \autoref{lem.pptVsFPTEasyComparison} and \autoref{lem.EasyLemmaBoundWhenAddingRoots} we obtain the following limiting statement.

\begin{corollary}
\label{cor.LimitingPPTStatement}
    Suppose $(V, \varpi)$ is a mixed characteristic $(0,p>0)$ complete DVR and $R = V\llbracket x_2, \dots, x_n\rrbracket$ has maximal ideal $\fm$.  Suppose $f \in \fm$ with corresponding $\overline{f} \in V/(\varpi)\llbracket x_2, \dots, x_n\rrbracket$.  Then 
    \[
        \lim_{e \to \infty} \ppt(f \in V[\varpi^{1/p^e}]\llbracket x_2, \dots, x_n \rrbracket) = \fpt(\overline{f} \in V/(\varpi)\llbracket x_2, \dots, x_n \rrbracket).
    \]
\end{corollary}
Unlike the case when $\fpt(\overline{f}) = a/p^e$, this limit does not always stabilize after finitely many steps, as the following example shows. 

\begin{example}
\label{exam.LimitingValueNonStabilizeExample}
    Let $p$ be an odd prime and $f = p^2 + x^2 \in \bZ_p \llbracket x \rrbracket$. For each $e>0$, set $R_e \coloneqq \bZ_p[p^{1/p^e}]\llbracket x \rrbracket$.  By \autoref{l-blow-up}(i), we obtain that
    \[
        \ppt(f \in R_e) \geq \fpt(y^{2p^e} + x^2 \in \bF_p\llbracket x,y\rrbracket)
    \]
    for all $e > 0$.  We compute the right side. 
    Note that $1/2$ can be written as $\sum_{i \geq 1} (p-1)/2p^i$ and $1/(2p^e)$ is $p^{-e}\sum_{i \geq 1} (p-1)/2p^i$, therefore \autoref{t-HerDiagHyp} guarantees that 
    \[
      \fpt(y^{2p^e} + x^2 \in \bF_p\llbracket x,y\rrbracket) = 1/(2p^e) + 1/2.
    \]
    Hence $\ppt(f \in R_e) > 1/2$ for all $e > 0$.  
    But $\fpt(\overline{f}) = 1/2$ and so the limiting value of $\ppt(f \in R_e)$ is \textit{never} achieved at any finite level.    

    Other examples work similarly, for instance $f = p^3 + y^3 + z^3 \in \bZ_p\llbracket y,z\rrbracket$ for $p \equiv_3 1$.
\end{example}

\begin{remark}
    We do not know if there is an example similar to that of  \autoref{exam.LimitingValueNonStabilizeExample} whose equation does not have an explicit $p$ in it (for instance, such that $\lct(f \in \ZZ_p[p^{1/p^e},x_2, \dots, x_n])$ is constant as $e$ varies).  For instance, if $p = 3$, then $\fpt(x^4+y^4+z^4+x^2y^2z^2 \in \bZ_p\llbracket x,y,z\rrbracket) = \frac{1}{2}$, while the $\lct$ of the same equation is equal to $3/4$ (see \cite{CantonHernandezSchwedeWitt.AtTheFP}). Hence from \autoref{lem.EasyLemmaBoundWhenAddingRoots}, we see that 
    \[
        {1 \over 2} = \lim_{e \to \infty} \ppt(x^4+y^4+z^4+x^2y^2z^2 \in \ZZ_p\llbracket p^{1/p^e}, x,y,z \rrbracket).
    \]
    But we do not know if this limit is achieved at a finite level.

    Similar potential examples to explore can be constructed from \cite[Example 4.5]{MustataTakagiWatanabeFThresholdsAndBernsteinSato} (for instance, $x^5 + y^4 + x^3y^2$ in characteristic $p \equiv_{20} 19$).
\end{remark}

Even without an explicit $p$ in the equation, we see that the $\ppt$ of common hypersurfaces can change quite dramatically based upon ramification:

\begin{example}[Yoshikawa]\label{yoshikawa}
    Yoshikawa proved that various hypersurface equations are perfectoid pure in \cite[Example 6.10]{yoshikawa2025computationmethodperfectoidpurity}.  For instance, set $A = \bZ_p\llbracket x,y,z \rrbracket, f = x^3+y^3+z^3$ and $R = A/(f)$. Yoshikawa proves that $R$ is perfectoid pure for $p \equiv_3 2$.  Hence by \autoref{prop.PPTWhenHypersurfaceIsPerfdPure}, we see that 
    \[
        \ppt(f \in \bZ_p\llbracket x,y,z \rrbracket) = 1.
    \] 
    But now as $\overline{f} = x^3+y^3+z^3 \in \bF_p[x,y,z]$ has $\fpt(\overline{f}) = 1-1/p$ for $p \equiv_3 2$ by \cite{BhattSinghThresholds}, we see that 
    \[
        \ppt(f \in \bZ_p[p^{1/p}]\llbracket x,y,z \rrbracket) = 1-1/p
    \]
    by \autoref{cor.EqualityOnTheNoseAfterRamifying}. We conclude that $\bZ_p[p^{1/p}]\llbracket x,y,z\rrbracket/(f)$ is \emph{not} perfectoid pure, thanks to \autoref{rem.WeakenHypothesesOfPPT}.
\end{example}

\begin{corollary}
\label{cor.EqualityOnTheNose}
    Suppose $R = V\llbracket x,y,z\rrbracket$ where $(V,\varpi)$ is a DVR containing a $p$-th root of $p$. Let $f$ be a homogeneous degree 3 equation in $x,y,z$ so that $\overline{f} \in V/(\varpi)\llbracket x,y,z\rrbracket$ defines a nonsingular elliptic curve $E$. Then 
    \[
        \ppt(f) = \fpt(\overline{f}) = \left\{ \begin{array}{ll} 1 & \text{if $E$ is ordinary,}\\ 1-1/p & \text{if $E$ is supersingular}. \end{array}\right.
    \]
\end{corollary}
\begin{proof}
    We always have $\ppt(f \in V\llbracket x,y,z\rrbracket) \geq \fpt(\overline{f}\in V/(\varpi)\llbracket x,y,z\rrbracket)$ by \autoref{lem.pptVsFPTEasyComparison}.
    
    If $E$ is ordinary, then $\fpt (\overline{f} \in V/(\varpi)\llbracket x,y,z\rrbracket) = 1$.  But $1$ is an upper bound on $\ppt$ and hence we have equality.  

    In the supersingular case, simply apply the fact that $\fpt(\overline{f}) = 1-1/p$ (\cite{BhattSinghThresholds}) and \autoref{cor.EqualityOnTheNoseAfterRamifying}.
\end{proof}

\subsection{Diagonal hypersurfaces}
In the spirit of \autoref{cor.EqualityOnTheNoseAfterRamifying}, we prove the following general statement which in particular shows that the plus-pure threshold of Calabi-Yau/Fermat type hypersurfaces involving a high enough $p$-th root of $p$ in the equation coincides with the $F$-pure threshold.

\begin{lemma} \label{l-homogeneoushyp_ramification}
Let $k$ be a perfect field of characteristic $p>0$.
Suppose $R = W(k)\llbracket x_2, \dots, x_n\rrbracket$.
Let $a \in \NN_{\geq 1}$ and define $R_a \coloneqq W(k)[p^{1/p^a}]\llbracket x_2,\dots, x_n \rrbracket$.
Let $s_1, \dots, s_n \in \ZZ_{>1}$, fix $f_a \coloneqq p^{s_1/p^a}+x_2^{s_2}+ \dots+x_n^{s_n} \in R_a$ and set $f_0 \coloneqq x_1^{s_1}+x_2^{s_2}+ \dots+x_n^{s_n} \in k\llbracket x_1, \dots, x_n\rrbracket$.
We follow the notation of \autoref{t-HerDiagHyp} and set ${1 \over s_i} = \sum_{e\geq 1}p^{-e}s_{i,e}$ with the $0 \leq s_{i,e} \leq p-1$ not all eventually zero\footnote{a non-terminating base-$p$ expansion of $1 \over s_i$}, assume that $L\coloneqq \min \{ e\geq 0 : \sum_{i=1}^n s_{i,e+1} \geq p \} < \infty$. 
Assume
\begin{itemize}
    \item[(a)] all the $s_i$'s are not powers of $p$, or
    \item[(b)] all the $s_i$'s are powers of $p$.
\end{itemize}
Then 
$$\ppt(f_a)\leq \fpt(f_0) \; \text{for all}\; a \geq L.$$
In particular, if $n=2$ or $s_1=\dots=s_n$, then $\ppt(f_a) = \fpt(f_0)$ for all $a \geq L$.
\end{lemma}

\begin{proof}
We first prove the lemma in Case (a).
The condition $s_i>1$ for every $i$ implies that there exists $e \geq L+1$ such that $s_{i,e} < p-1$. Indeed, if not, let $N$ be the minimum index such that $s_{i,e} = p-1$ for all $e \geq N+1$. Since $s_i>1$, then $N>0$. Then $s_i = \frac{p^N}{\sum_{e=1}^{N-1} p^{N-e}s_{i,e} + s_{i,N}+1}$ and the denominator is coprime with $p$, therefore $s_i$ cannot be an integer, which is a contradiction.
    
Since for every $i$ there exists $e \geq L+1$ such that $s_{i,e} < p-1$,
    \[
    \lfloor p^L/s_i \rfloor = \sum_{e=1}^L p^{L-e}s_{i,e}.
    \]
Denote by $a_L \coloneqq \sum_{e=1}^L \sum_{i=1}^n p^{L-e} s_{i,e} +1$.
    By \autoref{t-HerDiagHyp}, 
    \[
    \fpt(f_0) = \frac{1}{p^L} \left(\sum_{e=1}^L \sum_{i=1}^n p^{L-e} s_{i,e} +1\right)= a_L/p^L.
    \]
Let us compute $f_a^{a_L}$:
   \[
   f_a^{a_L} = \sum_{\ell_1+ \dots+\ell_n = a_L} c_{\ell} p^{s_1 \ell_1/ p^a} \cdot x_2^{s_2 \ell_2} \cdot \dots \cdot x_{n}^{s_{n} \ell_{n}};
   \]
   where $c_{\ell}$ are the multinomial coefficients $\binom{a_L}{\ell_1\, \ell_2 \, \dots\, \ell_n}$.
   We claim that in all the monomials in the above expressions there is at least one index $i$ such that $s_i \ell_i \geq p^L$.
Indeed, if this was not the case, then for all $i$ we would have $\ell_i \leq \lfloor p^L/s_i \rfloor$.
Therefore, by the computation above, $\sum_{i=1}^n  \lfloor p^L/s_i \rfloor < a_L$, which is a contradiction.
In particular,
$f_a^{a_L} \in (x_2^{p^L}, \dots, x_n^{p^L}, p^{p^L/p^a})$.
By \autoref{l-lemma2.3}(ii), if $a \geq L$, we conclude that
$f_a^{a_L/p^L} \in (x_2, \dots, x_n, p^{1/p^a})R_a^+$.

   We now prove Case (b). For all $i$, let $s_i = p^{e_i}$ for some integer $e_i \geq 1$. 
   First, let us order the indices $i_1, \dots, i_n$ such that $e_{i_1} \leq e_{i_2} \leq \dots \leq e_{i_n}$ (note that we cannot simply reorder the indices so that $e_1 \leq e_2 \leq \dots \leq e_n$ because the equation is not symmetric in the first variable).
   Then $L=e_{i_2}$ and, if we set 
   $a_L:= \sum_{e=1}^L \sum_{i=1}^n p^{L-e} s_{i,e} +1$, by \autoref{t-HerDiagHyp}, 
    \[
    \fpt(f_0) = \frac{1}{p^L} \left(\sum_{e=1}^L \sum_{i=1}^n p^{L-e} s_{i,e} +1\right)= a_L/p^L.
    \]
Now, let us compute $f_a^{a_L}$:
   \[
   f_a^{a_L} = \sum_{\ell_1+ \dots+\ell_n = a_L} c_{\ell} p^{s_1 \ell_1/ p^a} \cdot x_2^{s_2 \ell_2} \cdot \dots \cdot x_{n}^{s_{n} \ell_{n}};
   \]
   where $c_{\ell}$ are the multinomial coefficients $\binom{a_L}{\ell_1\, \ell_2 \, \dots\, \ell_n}$.
   We claim that in all the monomials in the above expression, there is at least one index $i$ such that $s_i \ell_i \geq p^L$.
    Indeed, if this was not the case then we would have $\ell_{i_j} < p^{e_{i_2}-e_{i_j}}$, which implies that $\ell_{i_2}=\dots=\ell_{i_n}=0$, since the $\ell_{i_j}$'s are integers and $p^{e_{i_2}-e_{i_1}}>\ell_{i_1}=a_L$.
    If $e_{i_1}=e_{i_2}$, this forces $\ell_{i_1}=0$, contradiction.
    
    If $e_{i_1}<e_{i_2}$, note that, for $j \geq 2$, $s_{i_j,e}=0$ for all $e \leq e_{i_2}$, therefore
    \[
    a_L= \sum_{e=e_{i_1}+1}^{e_{i_2}} (p-1)p^{e_{i_2}-e} +1 = p^{e_{i_2}-e_{i_1}},
    \]
    whence we derive the contradiction
    $$
    p^{e_{i_2}-e_{i_1}}>\ell_{i_1}=a_L=p^{e_{i_2}-e_{i_1}}.
    $$
    In particular, $f_a^{a_L} \in (x_2^{p^L}, \dots, x_n^{p^L}, p^{p^L/p^a})$.
By \autoref{l-lemma2.3}(ii), if $a \geq L$, we get
   $f_a^{a_L/p^L} \in (x_2, \dots, x_n, p^{1/p^a})R_a^+$.
   As for the  ``In particular'' part, we conclude by \autoref{l-blow-up}.
\end{proof}

\begin{remark}
With notation as in \autoref{l-homogeneoushyp_ramification} now suppose $L= \infty$.  If additionally either $n=2$ or $f_0$ is a homogeneous polynomial, then $\ppt(f_a)=\fpt(f_0)=\lct(f_a)$ for all $a \geq 0$ by \autoref{l-blow-up}.
\end{remark}

\begin{remark}
With notation as in \autoref{l-homogeneoushyp_ramification}, if $L< \infty$, we expect the equality $\ppt(f_a) = \fpt(f_0)$ for $a \geq L$ to hold also when $f_0$ is non-homogeneous. However, the arguments in \cite[Lemma 2.8]{CaiPandeQuinlanGallegoSchwedeTuckerpluspurethresholdscusplikesingularities} become more involved in higher dimension and so we do not work out the details.
\end{remark}

\begin{example}
    When $p$ does not appear in the equation, we can compute the plus-pure threshold by applying \autoref{cor.EqualityOnTheNoseAfterRamifying} to the case of diagonal hypersurfaces, even the non-homogeneous ones.
    Suppose $R = W(k)\llbracket x_1, \dots, x_n\rrbracket$, where $k$ is a perfect field of characteristic $p>0$. 
    Let $f= \sum_{i=1}^n x_i^{s_i}$ and write $\frac{1}{s_i} = \sum p^{-e}s_{i,e}$ so that $s_{i,e}$ are not eventually zero and $0 \leq s_{i,e} \leq p-1$.
Define $L\coloneqq \min \{ e\geq 0 : \sum_{i=1}^n s_{i,e+1} \geq p \}$.
Assume $L < \infty$.
Let $\overline{f}$ denote the image of $f$ in $R/(p) = k\llbracket x_1, \dots, x_n\rrbracket$.
Then, by \autoref{t-HerDiagHyp}, 
    \[
    \fpt(\overline{f}) = \frac{1}{p^L} \left(\sum_{e=1}^L \sum_{i=1}^n p^{L-e}s_{i,e} +1\right)=: a_L/p^L.
    \]
    Let $V \supseteq W(k)$ be a DVR containing some $p^L$-th root of $p$.  Then $\ppt(f \in V\llbracket x_2, \dots, x_n\rrbracket) = {a_L/p^L}$ by \autoref{cor.EqualityOnTheNoseAfterRamifying}.
\end{example}

\subsection{Computations at the finite level}\label{subsection:computations_at_the_finite_level.}

Since $\ppt(f) = \sup \{t \in \QQ_{>0} \;|\; f^t\notin \fm R^+\}$,
it is natural to check whether one can obtain information about the plus-pure threshold by studying the normalization of the ring $R[f^{1/p^e}]$. We obtain the following results: \autoref{thm:intermediate} and \autoref{thm:2}, which yield an upper bound of $1/p$ and $1-1/p$ respectively if there is tame ramification over $p$ in codimension one. Here are some explicit examples to keep in mind (see \autoref{cor:URLR_pthrootunity}, \autoref{rem:explicit} and \autoref{rem:URLR}):
 \begin{enumerate}
         \item Let $R=\ZZ_p[\zeta]\llbracket x_2,\dots,x_d\rrbracket$, where $\zeta$ is a primitive $p$-th root of unity. For any $f\in R$ admitting a $p$-th root modulo $(\zeta-1)^p$, we have $\ppt(f)\leq 1/p$.
          \item Let $R=\ZZ_p\llbracket x_2,\dots,x_d\rrbracket$. For any $f\in R$ admitting a $p$-th root modulo $p^2$, we have $\ppt(f)\leq 1-1/p$.
 \end{enumerate}

 Notice that the two examples above coincide for the special case $p=2$. Notice also that in the setting of (a), if $f$ admits a linear $p$-th root modulo $(\zeta-1)^p$ (for example $p=2$ and $f=(x_2+\dots+x_d)^2 + 4a$ for some $a\in R$), then $\ppt(f)=1/p$ $(=1/2$). Indeed, this follows by combining the above with the lower bound coming from the mod $p$ reduction (\autoref{lem.pptVsFPTEasyComparison}).

We now explain the connection of \autoref{thm:intermediate} and \autoref{thm:2} with \autoref{lem.EasyLemmaBoundWhenAddingRoots}. Suppose $R = V\llbracket x_2, \dots, x_n\rrbracket$ where $(V,\varpi)$ is a mixed characteristic $(0,p>0)$ complete DVR with uniformizer $\varpi$. Consider the subring $R^p$ of $R$ of elements that admit a $p$-th root modulo $\varpi$. Any $f\in R^p$ satisfies $\fpt(\overline{f})\leq 1/p$. Therefore, \autoref{lem.EasyLemmaBoundWhenAddingRoots} implies that $\ppt(f\in V[\varpi^{1/p}]\llbracket x_2, \dots, x_n\rrbracket)\leq 1/p$. \autoref{thm:intermediate} and \autoref{thm:2} can be viewed as providing analogous bounds for $f\in R^p$ without passing to $V[\varpi^{1/p^e}]\llbracket x_2, \dots, x_n\rrbracket$ under the stronger condition that $f$ admits a $p$-th root modulo certain higher powers of $\varpi$. For example, any $f\in \ZZ_p[\zeta]\llbracket x_2, \dots, x_n\rrbracket$ that admits a $p$-th root modulo $(\zeta-1)^p$ (which is a stronger condition than admitting a $p$-th root modulo $(\zeta-1)$), one has $\ppt(f\in \ZZ_p[\zeta]\llbracket x_2, \dots, x_n\rrbracket)\leq 1/p$.

 \begin{theorem}\label{thm:intermediate}
        Let $(S,\m)$ be a regular local ring of mixed characteristic $(0,p> 0)$ containing a primitive $p$-th root of unity and such that the irreducible components of $S/(p)$ are normal. Let $\{q_1,\dots,q_s\}$ be the prime divisors of $p\in S$. If $0\neq f\in \m$ is such that $S\to \overline{S[f^{1/p}]}$ ($\overline{*}$ is normalization) is tamely ramified (in particular, \'{e}tale) in codimension one over $q_i$ for some $1\leq i\leq s$, then $\ppt(f)\leq 1/p$.
    \end{theorem}

\begin{proof}
We may assume that $f$ does not have a $p$-th root in $S$: if it does, then $f^{1/p}\in \m S^+$ and we have $\ppt(f)\leq 1/p$. Set $A\coloneqq S[Y]/(Y^p-f)\simeq S[f^{1/p}]$, and $R$ to be the normalization of $A$.

Note that for a fixed $j$, $A$ is regular in codimension one over $q_j$ if and only if $\Gamma_{q_jS_{(q_j)}}(f)\leq 1$, where $\Gamma_{q_jS_{(q_j)}}(f)$ is the largest power $t$ of $q_j$ such that $f$ admits a $p$-th root in $S_{(q_j)}/q_j^tS_{(q_j)}$. To see this, suppose that $\Gamma_{q_jS_{(q_j)}}(f)\geq 2$, writing $f=h^p+aq_j^2$ for some $a,h\in S_{(q_j)}$, we have $S_{(q_j)}[f^{1/p}]\simeq S_{(q_j)}[Y]_{(q_j,Y-h)}/(Y^p-h^p-aq_j^2)$. Since $Y^p-h^p\in (q_j,Y-h)^2$, we see that $A$ is not regular in codimension one over $q_j$. Conversely, if $\Gamma_{q_jS_{(q_j)}}(f)=0$, then $Y^p-f\in \kappa(q_j)[Y]$ is irreducible so that $A$ is regular in codimension one over $q_j$. If $\Gamma_{q_jS_{(q_j)}}(f)=1$, writing $f=h^p+aq_j$ for some $a,h\in S_{(q_j)}$, with $a\notin (q_j)$, the isomorphism $S_{(q_j)}[f^{1/p}]\simeq S_{(q_j)}[Y]_{(q_j,Y-h)}/(Y^p-h^p-aq_j)$ tells us that $A$ is regular in codimension one over $q_j$.

Next, note that in our setup we have $\Gamma_{q_jS}(f)\geq p$ if and only if $\Gamma_{q_jS_{(q_j)}}(f)\geq p$, that is, there is no need to localize. The forward implication is obvious. Now assume $\Gamma_{q_jS_{(q_j)}}(f)\geq p$. In particular, $f$ has a $p$-th root in $\kappa(q_j)=\mathrm{Q}(S/q_jS)$. Since $S/q_jS$ is normal, we have $\Gamma_{q_jS}(f)\geq 1$. Write $f=h^p+aq_j$ for some $a,h\in S$. For some $h_1,h_2\notin q_jS$ and $b\in S_{(q_j)}$, we have in $S_{(q_j)}$:

\[h^p+aq_j= (h_1/h_2)^p+bq_j^p.\]

Multiplying across by $h_2^p$, we see that $(hh_2)^p-h_1^p\in q_jS$ and hence that $(hh_2-h_1)^p\in q_jS$. Thus, $hh_2-h_1\in q_jS$. Using this back in the above equation and noting that $\mathrm{ord}_{q_j}(p)\geq p-1$ (since $S$ contains a primitive $p$-th root of unity) yields $aq_j\in q_j^pS_{(q_j)}$. Hence $a\in q_j^{p-1}S_{(q_j)}\cap S=q_j^{(p-1)}S=q_j^{p-1}S$. This shows $\Gamma_{q_jS}(f)\geq p$.

We can now finish the proof. Let $i$ be such that $S\to R$ is tamely ramified in codimension one over $q_i$. By \cite[Theorem 1.1]{katzsridhar}, $f\notin q_iS$ and \[\Gamma_{q_iS_{(q_i)}}(f)\geq \dfrac{p}{p-1}\mathrm{ord}_{q_i}(p)\geq \dfrac{p}{p-1}(p-1)=p.\] From what we showed above $\Gamma_{q_iS}(f)\geq p$. Write $f=h^p+q_i^pb$ for some $b,h\in S$, $h\notin q_iS$. First suppose that $p$ is odd. Consider the following in $Q(A)$:
\begin{align*}
    &(f^{1/p})^p-h^p-q_i^pb=0\\
    & (f^{1/p}-h)^p-pc(f^{1/p}-h)-q_i^pb=0 \\
    & (f^{1/p}-h)^p-ucq_i^{p-1}(f^{1/p}-h)-q_i^pb=0
\end{align*}
for some $u,c\in A$. Setting $U\coloneqq f^{1/p}-h$ and $V\coloneqq q_i$ and dividing across by $V^p$, we get a deformation of an Artin--Schreier polynomial in $U/V$:
\begin{align*}
    &(U/V)^p-uc(U/V)-b=0.\\  
\end{align*}
In particular, $U/V\in R$ and $f^{1/p}=q_iU/V+h\in \mathfrak{m}R\subseteq \mathfrak{m}S^+$. Thus, $\ppt(f)\leq 1/p$. If $p=2$, one directly verifies that for $\alpha\coloneqq q_i^{-1}(f^{1/2}+h)\in Q(A)$, $\alpha$ is a root of the polynomial $X^2-hq_i^{-1}2X-b\in S[X]$. Thus $f^{1/2}\in \m S^+$ and $\ppt(f)\leq 1/2$.
\end{proof}

We record a special case of \autoref{thm:intermediate} below.

    \begin{corollary}\label{cor:URLR_pthrootunity}
        Let $(R,\m)$ be an unramified regular local ring of mixed characteristic $(0,p\geq 3)$ and $S\coloneqq R[X]/(\Phi_p(X))$ where $X$ is an indeterminate over $R$ and $\Phi_p(X)$ is the $p$-th cyclotomic polynomial. If $0\neq f\in S$ is a non-unit such that $S\to \overline{S[f^{1/p}]}$ ($\overline{*}$ is normalization) is \'{e}tale in codimension one over $p$, then $\ppt(f \in S)\leq 1/p$.
    \end{corollary}
\begin{proof}
   To apply \autoref{thm:intermediate} it suffices to note that $S$ is regular local. We confirm this. The identity $X^p-1\equiv_p (X-1)^p$ tells us that $S$ is local with maximal ideal $\mathrm{n}\coloneqq(\m,\zeta-1)$ where $\zeta$ is a primitive $p$-th root of unity. In $\mathbb{Z}[\zeta]$, $p=u(\zeta-1)^{p-1}$ where $u$ is a unit and hence $S$ is regular local. Moreover, $S/(\zeta-1)$ is regular and in particular normal.
\end{proof}

Note that an unramified regular local ring of mixed characteristic $p>0$ contains a primitive $p$-th root of unity if and only if $p=2$. Here is an unramified version of \autoref{thm:intermediate}:

\begin{theorem}\label{thm:2}
        Let $(S,\m)$ be an unramified regular local ring of mixed characteristic $(0,p>0)$. If $0\neq f\in \m$ is such that there exists $P\in \spec\overline{S[f^{1/p}]}$ ($\overline{*}$ is normalization) lying over $p$ with $S_{(p)}\to \overline{S[f^{1/p}]}_P$ \'{e}tale, then $\ppt(f)\leq 1-1/p$.
    \end{theorem}
    \begin{proof}
We first show that if $f$ is such that the condition in the statement is satisfied, then $f$ admits a $p$-th root modulo $p^2$ i.e. $f=h^p+p^2a$ for some $a,h\in S$ (the converse is also true, but it is not relevant to the proof). Suppose $f$ does not admit a $p$-th root modulo $pS$. Then since $S/pS$ is normal, it follows that $f$ does not admit a $p$-th root modulo $p$ in $S_{(p)}$ as well. Set $\mathcal{S}\coloneqq S\setminus(pS)$. It then follows that $S_{(p)}\to \mathcal{S}^{-1}S[f^{1/p}]$ is not \'{e}tale (there is a purely inseparable extension of residue fields). Since $S_{(p)}\to \mathcal{S}^{-1}\overline{S[f^{1/p}]}$ factors through the map $S_{(p)}\to \mathcal{S}^{-1}S[f^{1/p}]$, the former is not \'{e}tale. Now suppose that $f$ admits a $p$-th root modulo $p$, but not $p^2$, i.e., $f=h^p+ap$ and $a\notin pS$. Then $S_{(p)}[f^{1/p}]\simeq S_{(p)}[T]/(T^p-f)$ is local with uniformizer $f^{1/p}-h$. In particular $p\in (f^{1/p}-h)^2$ and hence $S_{(p)}\to \mathcal{S}^{-1}S[f^{1/p}]$ is not \'{e}tale. This again implies $S_{(p)}\to \mathcal{S}^{-1}\overline{S[f^{1/p}]}$ is not \'{e}tale.

\par Now assume $f$ has a $p$-th root modulo $p^2$ and write $f=h^p+p^2a$ for some $a,h\in S$. Suppose $\{Q_1,\dots,Q_n\}$ is the singular locus of $S[f^{1/p}]$ in codimension one. In other words, the $Q_i$ are the primes associated to the conductor $J$ of $S[f^{1/p}]$. From the isomorphism $S[f^{1/p}]\simeq S[T]/(T^p-f)$ and the form of $f$ it follows that there is a single codimension one prime $P\coloneqq(p,f^{1/p}-h)$ in $S[f^{1/p}]$ over $p$ and that $P$ is amongst the $Q_i$. Suppose $Q_1=P$. Set $R\coloneqq\overline{S[f^{1/p}]}$ and $A\coloneqq S[f^{1/p}]$. Since $A$ is Gorenstein, $R$ is reflexive in codimension one over $A$. Moreover, since $R$ satisfies $S_2$ over $A$, it is reflexive over $A$. Hence, it follows that $R$ can be identified with $\Hom_{A}(J,A)$. The latter can also be identified with the $A$-submodule of $Q(A)$ given by $(A:_{Q(A)}J)$. 

\par Now an injective map of ideals $I_1\to I_2$ in $A$ induces an injective map $\Hom_A(I_2,A)\to \Hom_A(I_1,A)$ (since the cokernel of  $I_1\to I_2$ is torsion). Thus, if $J_1$ is the $P$-primary component of $J$, there are injections $\Hom_{A}(P,A)\to \Hom_{A}(J_1,A) \to \Hom_{A}(J,A)= R$. This corresponds to the inclusion of $A$-submodules of $Q(A)$, $(A:_{Q(A)}P)\subseteq (A:_{Q(A)}J_1)\subseteq R$. Now the fact that $p^{-1}(f^{(p-1)/p}+hf^{(p-2)/p}+\dots+h^{p-1})\in (A:_{Q(A)}P)$ is easily verified. Thus, $p^{-1}(f^{(p-1)/p}+hf^{(p-2)/p}+\dots+h^{p-1})\in R$ and hence $f^{(p-1)/p}\in \m R\subseteq \m S^+$. This completes the proof.
\end{proof}

\begin{remark}\label{rem:explicit}
In the setting of \autoref{thm:intermediate}, an explicit characterization for the condition $S\to \overline{S[f^{1/p}]}$ being \'{e}tale in codimension one over $q_i$ is given by the numerical criterion
 \[\Gamma_{q_i}(f)\geq \dfrac{p}{p-1}\mathrm{ord}_{q_i}(p),\]
where $\Gamma_{q_i}(f)$ is the largest power $t$ of $q_i$ such that $f$ admits a $p$-th root in $S_{(q_i)}/q_i^tS_{(q_i)}$, see \cite[Theorem 1.1]{katzsridhar}.
\end{remark}

\begin{remark}\label{rem:URLR}
    With notation as in \autoref{thm:2}, note that the proof of \autoref{thm:2} shows that any $f\in S$ admitting a $p$-th root modulo $p^2$ satisfies the bound $\ppt(f)\leq 1-1/p$. Conversely, a computation shows that any $f$ of this form satisfies the hypothesis of the theorem, i.e., $S\to \overline{S[f^{1/p}]}$ is \'{e}tale in codimension one over $p$. Thus, if $S$ is $p$-complete, then by \autoref{prop.PPTWhenHypersurfaceIsPerfdPure}, any $p$-th power in $S$ has a ball of radius $1/p^2$ around it (under the $p$-adic metric) consisting of non perfectoid pure forms.
\end{remark}

\section{Extremal hypersurfaces and elliptic curves}
In the previous section, we noted that the plus-pure threshold is bounded below by the corresponding $F$-pure threshold. We showed that the plus-pure threshold decreases to eventually agree with the corresponding $F$-pure threshold after passing to a highly ramified DVR. In this section, we study several families of hypersurfaces for which, in the absence of any ramification, the plus-pure threshold no longer coincides with the corresponding $F$-pure threshold. 

\subsection{Extremal hypersurfaces}
Let $k$ be a perfect field of characteristic $p > 0$. For any homogeneous polynomial $\ov{f} \in k[x_1, \dots, x_n]$ that is reduced over the algebraic closure of $k$, \cite{KadyrsizovaKenkelPageSinghSmithVraciuWittLowerBoundsonFPureThresholdsandExtremalSingularities} determined a lower bound for $\fpt(\ov{f})$, denoting any $\ov{f}$ that attains this lower bound as an \emph{extremal singularity}. 
\begin{theorem}[\cite{KadyrsizovaKenkelPageSinghSmithVraciuWittLowerBoundsonFPureThresholdsandExtremalSingularities}, Theorem 1.1]
\label{Thm: ExtremalSingsPositiveChar}
    Let $\ov{f} \in k[x_1, \dots, x_n]$ be a homogeneous polynomial of degree $d \geq 2$ that is reduced over the algebraic closure of $k$. Then $$\fpt(\ov{f}) \geq \frac{1}{d-1},$$ with equality if and only if $d = p^e+1$ for some $e \geq 1$ and $\ov{f} =  \sum_{i=1}^n x_i^{p^e}L_i$, for $L_i$ linear forms. 
\end{theorem}
This theorem provides an explicit description of extremal hypersurfaces in positive characteristic. One can then ask about the plus-pure threshold of the corresponding polynomial $f \in R =  V\llbracket x_2, \dots, x_n\rrbracket$, for $(V,\varpi)$ any mixed characteristic $(0,p>0)$ DVR with uniformizer $\varpi$. We say that such an $f$ has an \emph{extremal singularity mod $p$} if $\ov{f} \in R/(\varpi) = (V/(\varpi))[x_2, \dots, x_n]$ has an extremal singularity in the sense of \autoref{Thm: ExtremalSingsPositiveChar}. We are defining $\ov{f}$ to be the image of $f$ under the map $R \to R/(\varpi)$, as in the statement of \autoref{lem.EasyLemmaBoundWhenAddingRoots}. Using this definition, we obtain a mixed characteristic analogue to the above theorem when the degree of $f$ is bounded by the order of roots of $p$ in $V$:
\begin{lemma}\label{lem:extremalMixed}
    Fix $e \geq 1$. Let $f \in V\llbracket x_2, \dots, x_n\rrbracket$ be a homogeneous polynomial of degree $d \geq 2$ such that $\ov{f}$ is reduced over the algebraic closure of $V/\varpi$, where $(V,\varpi)$ is a mixed characteristic $(0,p>0)$ complete DVR. Then
    \[
        \ppt(f) \geq \frac{1}{d-1}.
    \]  
     Further for $e \geq 0$, if $d \leq p^{e+1}$ and $p^{1/p^e} \in V$, then this is an equality if and only if $d = p^{e_0} + 1$ for some $1 \leq e_0 \leq e$ and $f$ has an extremal singularity mod $p$. 
\end{lemma}
\begin{proof}
    We have $\ppt(f) \geq \fpt(\ov{f})$ by \autoref{lem.pptVsFPTEasyComparison}, and $\fpt(\ov{f}) \geq \frac{1}{d-1}$ via \autoref{Thm: ExtremalSingsPositiveChar}. Now assume that $d \leq p^{e+1}$ and $p^{1/p^e} \in V$. If $f$ has an extremal singularity mod $p$ and is of degree $d = p^{e_0} + 1$, since $V$ contains a $p^{e_0}$-th root of $p$, by \autoref{cor.EqualityOnTheNoseAfterRamifying} and \autoref{Thm: ExtremalSingsPositiveChar}, 
    $$\ppt\left(f\right) = \fpt\left(\ov{f}\right) = \frac{1}{p^{e_0}}.$$
    So polynomials of this form achieve the desired bound. For the converse, since $\ppt(f) = \frac{1}{d - 1} \geq \fpt(\ov{f})$, it follows that $\fpt(\ov{f}) = \frac{1}{d-1}$. Thus, $\ov{f} = \sum_{i=2}^n x_i^{p^{e_0}}L_i$ is an extremal singularity via \autoref{Thm: ExtremalSingsPositiveChar}. In particular, $f$ has an extremal singularity mod $p$.  
\end{proof}

The following result establishes the curious fact that an extremal singularity mod $p$ may not be extremal in the sense of \autoref{lem:extremalMixed} when the coefficient DVR is \emph{not} ramified.  The proof is inspired by the arguments in \cite[Lemma 4.2, Corollary 4.3]{CaiPandeQuinlanGallegoSchwedeTuckerpluspurethresholdscusplikesingularities}.
    \begin{theorem}\label{thm:lowerbound}
 Let $R = W(k)\llbracket x, y, \mathbf{z}\rrbracket $, for $k$ a perfect field of characteristic $p > 0$, and $\mathbf{z} \coloneqq (z_1, \dots, z_n)$. Let $e \geq 1$ and $f' \in (p^{p^e}, x^{p^e}, y^{p^e}, \mathbf{z}^{p^e})R$ such that $f'=\sum_{i, j, k, \mathbf{h}} a_{i,j,k, \mathbf{h}}p^{i}x^{j}y^{k}\mathbf{z}^{\mathbf{h}}$, where $a_{i,j,k, \mathbf{h}} \in W(k)$ and for every monomial either $i \neq 0$ or $\mathbf{h} \neq \mathbf{0}$. Let $$(a,b) \in \{ (p^e+1, 0), (p^e, 1), (1, p^e), (0, p^e+1)\}.$$
 Then for $f =x^ay^b +x^by^a + f'$,  $f^{1/p^e} \notin (p,x, y, \mathbf{z})B$ for any BCM $R^+$-algebra $B$. In particular, $\ppt(f)>\frac{1}{p^e}$. 
\end{theorem}

\begin{proof}
    Assume, for contradiction, that $f^{1/p^e} \in (p, x,y, \mathbf{z})B$. 
    After picking some $p^e$-th roots, we define $g' \coloneqq \sum_{i, j, k, \mathbf{h}} a_{i,j,k, \mathbf{h}}^{1/p^e} p^{i/p^e}x^{j/p^e}y^{k/p^e}\mathbf{z}^{\mathbf{h}/p^e}$.
    Let $g \coloneqq x^{a/p^e}y^{b/p^e}+ x^{b/p^e}y^{a/p^e} + g'$ and observe that, since $f \in$ $(p^{p^e}, x^{p^e}, y^{p^e}, \mathbf{z}^{p^e})R$, we can deduce that $g \in (p, x,y,\mathbf{z})B$.
    Moreover, $(f^{1/p^e}-g)^{p^e} \in (p)B$, whence $(f^{1/p^e}-g) \in (p^{1/p^e})B$ by \autoref{l-lemma2.3}(i).
    Since we are assuming that $f^{1/p^e} \in (p,x,y, \mathbf{z})B$, and $B$ is big Cohen-Macaulay, we get that 
    \[
    (f^{1/p^e}-g) \in (p^{1/p^e})B \cap (p,x,y, \mathbf{z})B = (p, p^{1/p^e}x, p^{1/p^e}y, p^{1/p^e}\mathbf{z}).
    \]
    Therefore, there exist $\alpha, \beta, \gamma, \mathbf{d}$ in $B$ such that
    \begin{equation} \label{e-two sides}
        f^{1/p^e} = x^{a/p^e}y^{b/p^e}+ x^{b/p^e}y^{a/p^e} + g' + \alpha p^{1/p^e}x + \beta p^{1/p^e}y + \gamma p + \mathbf{d} p^{1/p^e} \mathbf{z}.
    \end{equation}
    Setting $I \coloneqq (p^{1+1/p^e}, x^{p^e}, y^{p^e}, \mathbf{z}^{1/p^e})B$,
    we have that $f \in I$. Indeed $x^ay^b+x^by^a \in I$ by our choices of $(a,b)$, while $f' \in (p^{p^e}, x^{p^e}, y^{p^e}, \mathbf{z}^{p^e})R \subseteq (p^{1+1/p^e}, x^{p^e}, y^{p^e}, \mathbf{z}^{1/p^e})B$.  For the remainder of the proof we split into two cases. First, suppose that $p^e < 3$, i.e. $p=2$ and $e=1$. Here $I \coloneqq (2^{3/2}, x^2, y^2, \mathbf{z}^{1/2})B$, and after squaring both sides of \autoref{e-two sides}, we claim that $2x^{3/2}y^{3/2} \in I$ regardless of the choice of $(a,b)$. 
    Indeed, the square of the left hand side is $f \in I$, which implies that the square of the right hand side is in $I$ as well.
    Consider the cross terms $2m_1m_2$ between two monomials $m_1$ and $m_2$: if either of the monomials is divisible by $2^{1/2}$, then $2m_1m_2 \in I$, whence $g^2 \in I$. Since, for every monomial in $g'$, either $i\neq 0$ or $\mathbf{h} \neq 0$, $g'^2 \in I$ and every cross term between a monomial in $g'$ and $x^{a/2}y^{b/2}$ or $x^{b/2}y^{a/2}$ is in $I$ as well.
    All in all, we conclude that $(x^{a/2}y^{b/2} + x^{b/2}y^{a/2})^2 \in I$, whence $2x^{(a+b)/2}y^{(a+b)/2} \in I$, which is the claim.
    
    Since $B$ is big Cohen-Macaulay and $(2, x,y)$ is a regular sequence, we deduce that $1 \in (2^{1/2}, x^{1/2}, y^{1/2}, \mathbf{z}^{1/2})B$, a contradiction.
    
     We now handle the case $p^e \geq 3$. As $f \in I$, by raising the two sides of \autoref{e-two sides} to the power $p^e$, we obtain that
     $$(x^{a/p^e}y^{b/p^e}+ x^{b/p^e}y^{a/p^e} + g' + \alpha p^{1/p^e}x + \beta p^{1/p^e}y + \gamma p + \mathbf{d} p^{1/p^e} \mathbf{z})^{p^e} \in I.$$
     Every cross term acquires a coefficient divisible by $p$, therefore, if the cross term involves at least a monomial divisible by $p^{1/p^e}$, it is automatically in $I$.
     Moreover, since we assume that every monomial in $g'$ has either $i \neq 0$ or $\mathbf{h} \neq 0$, $g'^{p^e} \in I$ and every cross term involving a monomial in $g'$ is automatically in $I$ as well.
     All in all, we have that $(x^{a/p^e}y^{b/p^e}+ x^{b/p^e}y^{a/p^e})^{p^e} \in I$.
     However,
    \[
    f_I \coloneqq (x^{a/p^e}y^{b/p^e}+ x^{b/p^e}y^{a/p^e})^{p^e} = \sum_{i=0}^{p^e} \binom{p^e}{i} x^{(ai+b(p^e-i))/p^e}y^{(bi+a(p^e-i))/p^e}.
    \]
    Consider the term for $i = p^{e-1}$.
    By \autoref{lem: binomialValuation}, we know that $v_p\binom{p^e}{p^{e-1}} = e - v_p(p^{e-1}) = e - (e - 1) = 1$. Thus, we have that $\binom{p^e}{p^{e-1}} = p u$, where $u \in W(k)$ is a unit.
    Moreover, as $p^e \geq 3$, both $$(ai+b(p^e-i))/p^e < p^e \; \text{ and }\; (bi+a(p^e-i))/p^e < p^e.$$
    Let $S \coloneqq R\llbracket p^{1/p^e},x^{1/p^e}, y^{1/p^e}, \mathbf{z}^{1/p^e}\rrbracket $ and set $s \coloneqq x^{1/p^e}, t \coloneqq  y^{1/p^e}$.
    Then $I \cap S= (p^{1+1/p^e}, s^{p^{2e}}, t^{p^{2e}}, \mathbf{z}^{1/p^e})$ and $f_I=\sum_{i=0}^{p^e} \binom{p^e}{i} s^{(ai+b(p^e-i))}t^{(bi+a(p^e-i))} \in I \cap S$.
    However there is at least a monomial in $f_I$---the one corresponding to $i = p^{e-1}$---whose terms in $s$ and $t$ have degree $<p^{2e}$ and such that the coefficient is divisible by $p$ and not $p^2$, a contradiction.
\end{proof}
\begin{remark}
\label{cor.EasyExtremalBound}
As a consequence of \autoref{thm:lowerbound}, let $f = p^{p^{e} + 1} + x_2^{p^e + 1} + \dots + x_n^{p^e + 1}$ or $f = p^{p^{e} + 1} + x_2^{p^e}x_3 + x_3^{p^e}x_2 + x_4^{p^e + 1} + \dots + x_n^{p^e + 1}$ in $W(k)\llbracket x_2, \dots, x_n\rrbracket$ for $n \geq 3$. These are both extremal singularities mod $p$, but $\ppt(f) > \fpt(\ov{f}) = \frac{1}{p^e}$ by \autoref{thm:lowerbound}. 
\end{remark}

\subsection{Elliptic curves}

\cite[\S 5]{CaiPandeQuinlanGallegoSchwedeTuckerpluspurethresholdscusplikesingularities} raises the question of computing plus-pure threshold of diagonal elliptic curves. As mentioned above, the question is answered for $f=x^3+y^3+z^3$ recently by Yoshikawa (\cite[Example 6.10]{yoshikawa2025computationmethodperfectoidpurity}) in characteristic $p \equiv_3 2$ over $W(k)$ while one obtains $1-{1 \over p}$ after appropriately ramifying the DVR as we saw in \autoref{yoshikawa}.  We consider the related example $f=x^3+y^3+p^3$ when the associated elliptic curve is supersingular.

\begin{theorem}\label{thm:ppt_elliptic}
Let $k$ be a perfect field of characteristic $p$ with $p \equiv_3 2$. Let $R \coloneqq W(k)\llbracket x, y\rrbracket $ with maximal ideal $\mathfrak{m}$ and $f = h(x,y) + p^3 \in R$, where $h(x,y) = xy(ux+vy)$ for some $u,v \in R$.  
    Then $$\ppt(f) \leq 1- 1/p^2.$$  Similarly, $\ppt(x^3+y^3+p^3 \in W(k)\llbracket x,y\rrbracket) \leq 1-1/p^2$.
\end{theorem}
\begin{proof}
It is sufficient to show that $f^{1 - 1/p^2} \in (p,x,y)R^+$. We write
    $$f =  (h(x,y)^{1/3})^3 + p^3 = \prod_{i=1}^3\left(h(x,y)^{1/3} + \zeta_i \cdot p\right)$$
    for $\zeta_i$ a certain sixth root of unity. We define $g_i \coloneqq h(x,y)^{1/3} + \zeta_i \cdot p$. It's immediately clear that for each $i$,
    $$g_i \in \left(h(x,y)^{1/3},p\right)R^+.$$
    Thus, by \autoref{l-lemma2.3}
    $$g_i^{1/p^2} \in \left(h(x,y)^{1/3p^2},p^{1/p^2}\right)R^+.$$
    Considering the product, 
    $$f^{\frac{p^2 - 1}{p^2}} = \left(g_1g_2g_3\right)^{\frac{p^2 - 1}{p^2}} \in \left(h(x,y)^{1/3p^2},p^{1/p^2}\right)^{3(p^2 - 1)}R^+.$$
    It is thus sufficient to show that $\left(h(x,y)^{1/3p^2},p^{1/p^2}\right)^{3(p^2 - 1)}R^+ \subset (p,x,y)R^+$. Indeed if we expand the ideal product, we see that
    $$\left(h(x,y)^{1/3p^2},p^{1/p^2}\right)^{3(p^2 - 1)} = \left(h(x,y)^{\alpha/3p^2}p^{\beta/p^2} \ \bigg| \ \alpha + \beta = 3(p^2 - 1)\right).$$
    If $\beta \geq p^2$, $p^{\beta/p^2} \in (p,x,y)R^+$. By blowing up the ideal $(x,y)$, and using that $h(x,y) \in (x,y)^3$, we see that $h(x,y)$ has an lct of at most $2/3$. Thus, as $\ppt$ is bounded above by the lct, if $\alpha/3p^2 \geq 2/3$, $h(x,y)^{\alpha/3p^2} \in (p,x,y)R^+$. As $\alpha + \beta = 3(p^2 - 1)$, this bound is equivalent to enforcing that $\beta \leq p^2 - 3$. This leaves only two generators for which we need  to check the inclusion: when the pair $(\alpha,\beta)$ is of the form $(2p^2 - 2, p^2 - 1)$ [case (1)] or $(2p^2 - 1, p^2 - 2)$ [case (2)]. 
    \begin{enumerate}[label = (\arabic*)]
        \item It is sufficient to show that 
        \[ 
            (xy(ux + vy))^{\frac{2p^2 - 2}{3p^2}} \cdot p^{\frac{p^2 - 1}{p^2}} \in \left(p,x,y\right)R^+.
        \]
        Using the fact that $p,x,y$ forms a regular sequence on $R^+$, this is equivalent to checking that
        $$\left(ux + vy\right)^{\frac{2p^2 - 2}{3p^2}} \in \left(p^{1/p^2},x^{\frac{p^2 + 2}{3p^2}}, y^{\frac{p^2 + 2}{3p^2}}\right)R^+.$$
        By \autoref{l-lemma2.3}(ii) we can clear $p^2$ from the denominators in our exponents, and thus it is sufficient to show that 
        $$\left(ux + vy\right)^{\frac{2p^2 - 2}{3}} \in \left(p,x^{\frac{p^2 + 2}{3}}, y^{\frac{p^2 + 2}{3}}\right)R^+.$$
        We note that for any prime $p$, $p^2 \equiv_3 1$, and thus, all powers above are integer powers. Thus we can take the binomial expansion of the polynomial on the left hand side:
        $$\left(ux + vy\right)^{\frac{2p^2 - 2}{3}} = \sum_{a + b = \frac{2p^2 - 2}{3}} \binom{a+b}{b} u^a v^b x^ay^b.$$
        We note that if $a + b = \frac{2p^2 - 2}{3}$, the only choice of integers $a$ and $b$ for which $a,b < \frac{p^2 + 2}{3}$ is precisely when $a = b = \frac{p^2 - 1}{3}$. It follows then that all choices of $a,b$ outside of this lead to $x^ay^b \in \left(p,x^{\frac{p^2 + 2}{3}}, y^{\frac{p^2 + 2}{3}}\right)$ as desired. Thus it is sufficient to check that, for $a=b = \frac{p^2 - 1}{3}$, the binomial coefficient of the last remaining monomial, $\binom{2a}{a}$, is divisible by $p$. As it turns out, this is true precisely when $p \equiv_3 2$; see \autoref{lem: LucasThmCor}. We note that the combinatorial identity is to be expected, as when $p \equiv_3 1$, the given elliptic curve is ordinary and thus has $\ppt = 1$. 

        \item It is sufficient to show that 
         \[
         (xy(ux + vy))^{\frac{2p^2 - 1}{3p^2}} \cdot p^{\frac{p^2 - 2}{p^2}} \in \left(p,x,y\right)R^+.
         \]
         We proceed similarly to the previous case. Using the fact that $p,x,y$ forms a regular sequence on $R^+$, this is equivalent to checking that
        \[
        \left(ux + vy\right)^{\frac{2p^2 - 1}{3p^2}} \in \left(p^{2/p^2},x^{\frac{p^2 + 1}{3p^2}}, y^{\frac{p^2 + 1}{3p^2}}\right)R^+.
        \]
        By \autoref{l-lemma2.3}(ii), this is equivalent to checking that
         \[
        \left(ux + vy\right)^{\frac{2p^2 - 1}{3p}} \in \left(p^{2/p},x^{\frac{p^2 + 1}{3p}}, y^{\frac{p^2 + 1}{3p}}\right)R^+.
        \]
        Since $p \equiv_3 2$, by \autoref{l-magic}, the floor of $\frac{2p^2 - 1}{3p}$ is $\frac{2p-1}{3}$. Therefore, if we prove that
        \[
        \left(ux + vy\right)^{\frac{2p - 1}{3}} \in \left(p^{2/p},x^{\frac{p^2 + 1}{3p}}, y^{\frac{p^2 + 1}{3p}}\right)R^+,
        \]
        then we are done. Note that $\left(ux + vy\right)^{\frac{2p - 1}{3}} = \sum_{i+j = \frac{2p - 1}{3}} c_{i,j} u^iv^jx^iy^j$ for some $c_{i, j} \in \ZZ_{>0}$. Clearly, all the monomials in the sum have either $i$ or $j$ $\geq (2p-1)/6$. Since $p \equiv_3 2$  and $i$ and $j$ are integers, the ceiling of $(2p-1)/6$ is $(p+1)/3$, which is always $\geq \frac{p^2 + 1}{3p}$. Therefore all the monomials in the sum indeed lie in the ideal$\left(p^{2/p},x^{\frac{p^2 + 1}{3p}}, y^{\frac{p^2 + 1}{3p}}\right)$.         
    \end{enumerate}
    This finishes the proof for the equation $f=h(x,y)+p^3$.

   As for the equation $f=x^3+y^3+p^3$, since $p \equiv_3 2$, there exists an \'etale extension $R' = W(k')\llbracket x,y\rrbracket \supseteq W(k)\llbracket x,y\rrbracket$ containing third roots of unity.  By \autoref{lem.EtaleExtensionPPT} we see that $\ppt\big(f \in R'\big) = \ppt\big(f \in R\big)$.  Hence we may assume that $R := R'$ contains a third root of unity $\xi$.  
    
    Now, $p^3 + x^3 + y^3 = p^3 + (x+y)(x+\xi y)(x + \xi^2 y)$.  Consider the automorphism $\phi : R \to R$ sending $x \mapsto x+y$ and $y \mapsto x + \xi y$ (this is an isomorphism as $p \neq 3$).  We see that $\phi^{-1}(p^3+x^3+y^3) = p^3 + xy(-\xi x + (\xi + 1)y)$.  In particular, as the $\ppt$ of the right side is $\leq 1-1/p^2$, we see that 
    \[
        \ppt(p^3+x^3+y^3) \leq 1-1/p^2
    \] 
    as well.
\end{proof}

\begin{remark}
    The automorphism argument at the end applies to many other equations as well.  It perhaps is worth noting that all the elliptic curves defined by equations of the form $z^3 + xy(x+\lambda y) \in \overline{k}\llbracket x,y,z\rrbracket$ are all isomorphic for any nonzero $\lambda \in k$.  Indeed after replacing $y$ by $\lambda y$, one gets the equation $z^3 + {1 \over \lambda}xy(x+y)$ which defines the same variety as $\lambda z^3 + xy(x+y)$.  But then $\lambda$ can be absorbed into $z$.  It would be interesting to study the $\ppt$ of expressions of the form
    \[  
        py^2 + x(x+y)(x-\lambda y).
    \]
\end{remark}

We note that the $2$-adic version of the diagonal elliptic curve gives an explicit example of a polynomial for which the plus-pure threshold differs from the log canonical threshold as well as the corresponding $F$-pure threshold. This answers the generalization of \cite[Question 5.2]{CaiPandeQuinlanGallegoSchwedeTuckerpluspurethresholdscusplikesingularities} immediately following it in characteristic $2$. 

\begin{remark} \label{rmk: 2adic elliptic curve PPT bounds}
 By \autoref{thm:lowerbound} and \autoref{thm:ppt_elliptic}, we get that for $f=x^3+y^3+2^3 \in \ZZ_2\llbracket x,y \rrbracket$ and $f_0 = x^3+y^3+z^3 \in \FF_2\llbracket x,y,z\rrbracket$,
 $$\fpt(f_0)=1/2<\ppt(f)\leq 3/4 < 1=\lct(f).$$
\end{remark}
Though expected, it is unknown to the authors whether such bounds hold for any $p > 2$ such that $p \equiv_3 2$.
\begin{question} Let $k$ be a perfect field of characteristic $p \equiv_3 2$. Let $f = x^3 + y^3 + p^3 \in W(k) \llbracket x,y \rrbracket$ and $f_0 = x^3 + y^3 + z^3 \in k\llbracket x,y,z\rrbracket$. Is it true that
   $$\fpt(f_0) = 1 - \frac{1}{p} < \ppt(f) \leq 1 - \frac{1}{p^2}?$$
Note that it suffices to show the inequality $1-\frac{1}{p} < \ppt(f)$.
\end{question}

\begin{remark}
\label{rem.PTimesJumpingNumber}
Consider $f = p^3 + x^3+y^3 \in W(k)\llbracket x,y \rrbracket = R$ when $p = 2$.  By \autoref{thm:lowerbound} and \autoref{thm:ppt_elliptic}
we see that $\ppt(f) \in \left(\frac{1}{2},\frac{3}{4}\right]$, so $p \cdot (\ppt(f)) = 2 \cdot (\ppt(f)) \in \left(1,\frac{3}{2}\right]$. Suppose for a contradiction that $p \cdot (\ppt(f))$ were a jumping number of the associated $+$-test ideal.  
In that case, we would have for $1 \gg \epsilon > 0$ that 
\[
    \tau_+(f^{p \cdot \ppt(f) + \epsilon}) \neq \tau_+(f^{p \cdot \ppt(f) - \epsilon})
\]
where here $\tau_+$ denotes the test ideal of \cite{MaSchwedeSingularitiesMixedCharBCM} associated to the perfectoid BCM algebra $\widehat{R^+}$ (see also \cite{BMPSTWW2} for comparisons with other theories).
Now, for any rational number $t$, write $t = \lfloor t \rfloor + \{t \}$, with $\{t \}$ the fractional part.  Since $\tau_+(f^t) = f^{\lfloor t \rfloor} \tau_+(f^{\{t\}})$, we immediately see the fractional part of $p \cdot (\ppt(f))$ is a jumping number as well.  
But $\{ p \cdot \ppt (f)\} \in \left(0,\frac{1}{2}\right]$.  However, the first jumping number, $\ppt(f)$,  is strictly greater than $\frac{1}{2}$, a contradiction.

This shows that the analog of \cite[Lemma 3.1(1)]{BlickleMustataSmithDiscretenessAndRationalityOfFThresholds} \textit{fails} in mixed characteristic; in particular, $p$ times a jumping number need not be a jumping number.   This is particularly concerning since this property plays a key role in proving the rationality (and sometimes discreteness) of the $F$-jumping numbers and in particular, in proving the rationality of the $F$-pure threshold.
\end{remark}

\subsection{Non-reduced modulo 
$p$ reduction}
The following result---in the particular case of $(p,n,e) = (3,1,1)$ ---partially answers \cite[Question 5.1]{CaiPandeQuinlanGallegoSchwedeTuckerpluspurethresholdscusplikesingularities}. It also shows that \cite[Proposition 4.6]{CaiPandeQuinlanGallegoSchwedeTuckerpluspurethresholdscusplikesingularities} (with $a=p$ in its notation) does not extend for any exponent of an odd prime and in any dimension.

\begin{theorem}\label{TuckersBane}
    Let $k$ be a perfect field of characteristic $p>2$. Let $f = p^{p^e}+ x_2^{p^e} + \dots + x_n^{p^e} \in R\coloneqq W(k)\llbracket x_2, \dots, x_n\rrbracket $ and $f_0=x_1^{p^e} + \dots + x_n^{p^e} \in R_0\coloneqq k\llbracket x_1, \ldots, x_n\rrbracket $.
    Then $f^{1/p^e} \notin (p, x_2, \dots, x_n)B$ for any BCM $R^+$-algebra $B$.
    In particular, $\ppt(f) > \fpt(f_0) = \frac{1}{p^e}$.
\end{theorem}

Compare with \autoref{cor.PthPowersRamified} and \autoref{l-homogeneoushyp_ramification} for the case of a ramified DVR.  
\begin{proof}
Since $f_0 = (x_1+\ldots +x_n)^{p^e}$ is a product of $p^e$ linear forms, it is clear that $\fpt(f_0)=\frac{1}{p^e}$. To establish the assertion on the plus-pure threshold, we proceed by contradiction. 

Suppose that the assertion is false. Then $f^{1/p^e} \in (p,x_2, \dots, x_n)B$, where $B$ is a BCM $R^+$-algebra. Let $g \coloneqq \sum_{i=2}^n x_i + p \in (p,x_2, \dots, x_n)B$. Clearly $(f^{1/p^e} - g)^{p^e} \in (p)B$. So 
    $$f^{1/p^e} - g \in (p^{1/p^e})B \cap (p,x_2, \dots, x_n)B = (p,p^{1/p^e}x_2, \dots, p^{1/p^e}x_n)B.$$
    Therefore, there exist $a_2, \dots, a_n,b\in B$ such that 
    $$f^{1/p^e} =  \sum_{i=2}^n a_ip^{1/p^e}x_i + bp + g =  \sum_{i=2}^n x_i(a_ip^{1/p^e} + 1) +p(b + 1).$$
    We now take the $p^e$-th power on both sides modulo the ideal $$I \coloneqq (p^{e + 1 + 1/p^e},x_2^{p^e}, x_3, \dots, x_n)B.$$ 
    Note that the left hand side is $(f^{1/p^e})^{p^e} = f \equiv_I 0$ since $p>2$. Let $$h \coloneqq (\sum_{i=2}^n x_i(a_ip^{1/p^e} + 1) +p(b + 1))^{p^e}.$$
    Since $x_3, \dots, x_n \in I$, we have that $h \equiv_I (x_2(a_2 p^{1/p^e} + 1) +p(b + 1))^{p^e}$. Expanding the binomial, we have that 
    $$ h \equiv_I \sum_{i = 0}^{p^e} \binom{p^e}{i}p^i x_2^{p^e-i}(p^{1/p^e}a_2+1)^{p^e-i}(b+1)^i.$$
Notice that $\binom{p^e}{i}p^i x_2^{p^e-i}(p^{1/p^e}a_2+1)^{p^e-i}(b+1)^i$ has a factor of $p^{e+2}$ for $i \geq e + 2$. Hence

    $$h \equiv_I \sum_{i = 0}^{e + 1} \binom{p^e}{i}p^i x_2^{p^e-i}(p^{1/p^e}a_2+1)^{p^e-i}(b+1)^i.$$

    The first term of $h$ vanishes modulo $I$ since it has a factor of $x_2^{p^e}$. The $(e + 1)$-th term also vanishes modulo $I$ since it has a factor of $\binom{p^e}{e+1}p^{e+1}$ and $p$ divides $\binom{p^e}{e+1}$. Thus, 

    $$h \equiv_I \sum_{i = 1}^{e} \binom{p^e}{i}p^i x_2^{p^e-i}(p^{1/p^e}a_2+1)^{p^e-i}(b+1)^i.$$

    Now we show that the terms $\binom{p^e}{i}p^i x_2^{p^e-i}(p^{1/p^e}a_2+1)^{p^e-i}(b+1)^i$ for $2 \leq i \leq e$ vanish modulo $I$. To show this, we prove that $p^{e + 2 - i}$ divides $\binom{p^e}{i}$. Thus, it is enough to show that $e + 2 - i \leq v_p\binom{p^e}{i}$. Since $v_p\binom{p^e}{i} = e - v_p(i)$ by \autoref{lem: binomialValuation}, we proceed to prove that $v_p(i) \leq  i - 2$ for all $i \geq 2$.

    Observe that for $n \geq 1$, we have that $v_p(p^n) = n \leq p^n - 2$, since $p > 2$. Now let $i \geq 2$. We have that $i \in [2, p)$ or $i \in [p^c, p^{c+1})$ for some $c \geq 1$. If $i \in [2, p)$, then $0 = v_p(i) \leq i -2$ so that this term lies in $I$. Similarly, if $i \in [p^c, p^{c+1})$ for some $c \geq 1$, then $v_p(i) \leq c$. Furthermore, since $v_p(p^c) = c$, we have that $c \leq p^c - 2$. Finally, $p^c - 2 \leq i - 2$ since $p^c \leq i$. From these three inequalities we conclude that $v_p(i) \leq i - 2$. Thus $v_p(i) \leq  i - 2$ for all $i \geq 2$ and therefore all terms of $h$ other than the first term lie in $I$. So, we get that

    $$h \equiv_I p^{e+1}x_2^{p^e-1}(p^{1/p^e}a_2+1)^{p^e-1}(b+1) \equiv_I p^{e+1}x_2^{p^e-1}(b+1).$$ Therefore
    \[p^{e+1}x_2^{p^e-1}(b+1) \in I = (p^{e + 1 + 1/p^e},x_2^{p^e}, x_3, \dots, x_n)B.\]
    Since B is a BCM $R^+$-algebra, we get
    $$b+1 \in (p^{1/p^e}, x_2, \dots, x_n)B.$$
    So, we can write
    $$b+1 =p^{1/p^e}\alpha +\sum_{i=2}^n x_i\beta_i \quad \text{for some }\alpha,\beta_2, \dots, \beta_n \in B.$$
    Plugging this in the expression for $f^{1/p^e}$, we get 
     $$f^{1/p^e} = \sum_{i=2}^n x_i(a_ip^{1/p^e} + 1) +p(p^{1/p^e}\alpha +\sum_{i=2}^n x_i\beta_i).$$
    Now we take the $p^e$-th power on both sides modulo the ideal $(x_2, \dots, x_n)B$ to get that
    \begin{align*}
    p^{p^e} &= p^{p^e+1}\alpha^{p^e} &  \in B/(x_2, \dots, x_n)B, \\
    \implies 1 &= p\alpha^{p^e} &  \in B/(x_2, \dots, x_n)B,\\
    \implies 1 &= 0 &  \in B/(p,x_2, \dots, x_n)B.
    \end{align*}
   But then $B=(p,x_2, \dots, x_n)B$, contradicting the assumption that $B$ is a BCM $R^+$-algebra. 
\end{proof}

\begin{remark}
    \autoref{yoshikawa} shows that the plus-pure threshold can vary significantly depending on the coefficient DVR. A counterpart of this phenomenon involving ``$p$-terms" is as follows: we have $$\ppt(f=x^3+3^3\in \ZZ_3\llbracket x\rrbracket)>1/3$$ by \autoref{TuckersBane}. Note that the DVR $\ZZ_3[\zeta]$, for $\zeta$ a primitive $3$-rd root of unity, has uniformizer $\varpi=\zeta-1$. We claim that the ppt of the analogous form $$\ppt(f'=x^3+\varpi^3\in \ZZ_3[\zeta]\llbracket x\rrbracket)=1/3.$$ This is because $\mathrm{ord}_{\varpi}(p)=2$, so \autoref{rem:explicit} and \autoref{thm:intermediate} yield $$1/3=\fpt(\overline{f'})\leq \ppt(f')\leq 1/3.$$
\end{remark}

\bibliographystyle{skalpha}
\bibliography{MainBib}
\end{document}